\documentclass[leqno, 12pt]{amsart}
\usepackage[utf8]{inputenc}
\usepackage[english]{babel}
\usepackage{amsthm}
\usepackage{amsmath}
\usepackage{amssymb}
\usepackage{mathtools}
\usepackage{physics}
\usepackage[margin=1in]{geometry}
\usepackage[unicode=true,pdfusetitle,
 bookmarks=true,bookmarksnumbered=false,bookmarksopen=false,
 breaklinks=false,pdfborder={0 0 1},backref=false,colorlinks=true]
 {hyperref}
\usepackage[svgnames]{xcolor}

\definecolor{color1}{RGB}{27,158,119}
\definecolor{color2}{RGB}{217,95,2}
\definecolor{color3}{RGB}{117,112,179}
\definecolor{color4}{RGB}{231,41,138}

\hypersetup{
  pdftitle={Ground states on open books},
  pdfauthor={Stefan Le Coz  and Boris Shakarov},
  pdfsubject={},
  pdfkeywords={}, 
  colorlinks=true,
  linkcolor=DarkBlue  ,          %
  citecolor=DarkRed,        %
  filecolor=DarkMagenta,      %
  urlcolor=DarkGreen,           %
}

\newtheorem{theorem}{Theorem}[section]
\newtheorem{lemma}[theorem]{Lemma}
\newtheorem{proposition}[theorem]{Proposition}

\newtheorem{remark}[theorem]{Remark}

\theoremstyle{definition}
\newtheorem{definition}[theorem]{Definition}

\def\R{\mathbb R}
\def\eps{\varepsilon}
\def\T{\mathbb T}

\newcommand{\Graph}{\mathcal{G}}
\newcommand{\GraphTimesZeroL}{\mathcal{G}\times[0,L]}
\newcommand{\GraphTimesZeroOne}{\mathcal{G}\times[0,1]}
\newcommand{\Lib}{\mathcal{L}}

\DeclareMathOperator{\supp}{supp}
\let\abs\relax
\DeclarePairedDelimiter{\abs}{\lvert}{\rvert}
\DeclarePairedDelimiterX{\dual}[2]{\langle}{\rangle}{#1, #2}
\DeclarePairedDelimiterX{\scalar}[2]{\lparen}{\rparen}{#1, #2}

\title[Ground States on Open Books and Dimension Reduction]{Ground States for the Nonlinear Schr\"odinger Equation on Open Books and Dimensional Reduction to Metric Graphs}

\author[S. Le Coz]{Stefan Le Coz}
\author[B. Shakarov]{Boris Shakarov}

\address{Stefan Le Coz and Boris Shakarov, Univ Toulouse, INUC, UT2J, INSA Toulouse, TSE, CNRS, IMT, Toulouse, France.}

\email{stefan.le-coz@math.univ-toulouse.fr}

\email{boris.shakarov@math.univ-toulouse.fr}
\thanks{The work of S. L. C. and B. S. is 
  partially supported by ANR-11-LABX-0040-CIMI and the ANR project NQG ANR-23-CE40-0005}

\date{\today}

\subjclass[2010]{35Q55 (35A15, 35R02)}

\date{\today}
\keywords{nonlinear Schr\"odinger equation, standing waves, ground state, nonlinear quantum graphs, open books}

\numberwithin{equation}{section}

\begin{document}

\begin{abstract}
In this work, we study the dimensional reduction of stationary states in the shrinking limit for a broad class of two-dimensional domains, called open books, to their counterparts on metric graphs. An open book is a two-dimensional structure formed by rectangular domains sharing common boundaries. We first develop a functional-analytic framework suited to variational problems on open books and establish the existence of solutions as constrained action minimizers.

For graph-based open books (i.e., those isomorphic to the product of a graph with an interval) we prove the existence of a sharp transition in the dimensionality of ground states. Specifically, there exists a critical transverse width: below this threshold, all ground states coincide with the ground states on the underlying graph trivially extended in the transverse direction; above it, ground states become genuinely two-dimensional.
\end{abstract}

\maketitle

\section{Introduction}

We consider the nonlinear Schrödinger equation set on a structure $\Lib$, which we will refer to as an \emph{open book}. Open books are popular structures in various areas of mathematics such as in contact topology or algebraic geometry (see e.g. \cite{Et06,Gi05}). They appear naturally in various areas of physics, e.g. in metamaterial designs (see \cite{LaTaCh22}).
They can also be considered as a form of generalized waveguides.
A prototypical example of the type of structure that we want to be able to treat is represented in Figure \ref{fig:proto}.  We will use the following definition.

\begin{figure} 
    \centering
\includegraphics[width=0.7\linewidth]{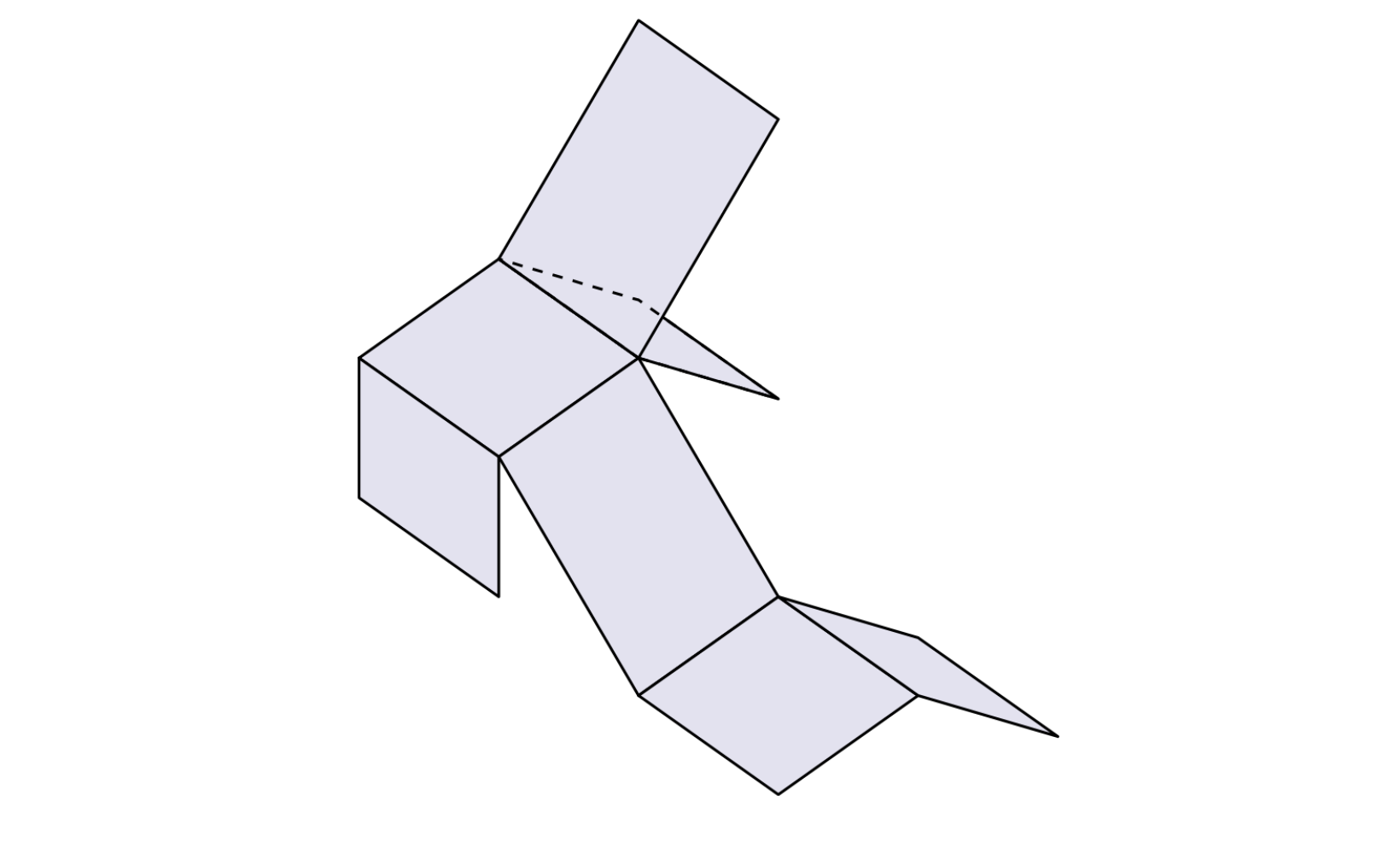}
    \caption{An open-book with seven pages}
    \label{fig:proto}
\end{figure}

\begin{definition}
    An \emph{open book} $\Lib$ is a collection of $1$-$d$ manifolds called \emph{bindings} $\mathcal{B} = (B_j)_{j\in \mathcal J\subset \mathbb N}$ and $2$-$d$ manifolds called      \emph{pages} $\mathcal{P} = (P_k)_{k\in \mathcal K\subset \mathbb N}$. A binding $B_j\in\mathcal{B}$ is characterized by a length $L_j\in(0,\infty]$ and is isometric to the interval $[0,L_j]$ ($[0,\infty)$ if $L_j=\infty$). A page $P_k\in\mathcal{P}$ is characterized by two lengths $L_k^1,L_k^2\in(0,\infty]$ and is isometric to the rectangle $[0,L_k^1]\times [0,L_k^2] $ (replacing $[0,L_k^{1,2}]$ by $[0,\infty)$ whenever $L_k^{1,2}=\infty$). For each page $P_k$, there exist bindings $(B_k^j)_{j\in\{0,\dots,J\}}\subset \mathcal{B}$, with $B_j\neq B_k$ if $j\neq k$, such that the boundary $\partial P_k$ of $P_k$ verifies 
    \[
    \partial P_k=\bigcup_{j\in\{0,\dots,J\}}B_k^j.
    \]
    Here, $J=3$ if $L_k^1+L_k^2<\infty$, $J=2$ if $L_k^1=\infty$ or $L_k^2=\infty$, $J=1$ if $L_k^1=L_k^2=\infty$. 
    We use the notation $P\sim B$ to express the fact that $P$ is one of the pages incident to the binding $B$. 

\end{definition}

\begin{remark}
With the above definition, where the bindings of a page are all different, every single page can be embedded in $\mathbb R^2$. The definition can be relaxed to allow for cylindrical pages (i.e., pages having two identical non-consecutive bindings), toroidal pages (i.e., pages having two by two non-consecutive bindings identical), or even M\"obius strip-type pages. Indeed, these cases can be included in our definition by the introduction of artificial bindings, cutting the page into two (or four) new pages, and having all bindings different. On the other hand, we cannot relax the definition to allow for conical pages (i.e., two consecutive bindings are identical).
    Observe that an infinite strip (such as the one considered in \cite{LeSh24}) is formed of two pages with one infinite length connected by their (transversal) finite length binding. Similarly, a half-plane is made of two quarter-plane pages.
\end{remark}

Various non-equivalent definitions of open books or stratified structures are used depending on the context. The definition that we adopt in this work is tailored to our purposes. On the one hand, it is more restrictive than definitions used in other contexts, such as contact geometry (see e.g. \cite{Et06}), as we are working only with $2$-$d$ pages isometric to rectangles. On the other hand, 
the fact that we do not embed our books in $\mathbb R^d$ (as is done e.g. in the context of stratified sets in \cite{NiPe04} or for the spectral analysis in \cite{AkKu24,Co20,CoKu20}) allows for extra flexibility in the analysis, as we do not have to take into consideration geometric features of the pages such as curvature. That flexibility is reminiscent of the flexibility allowed by quantum graphs by concentrating the main features of the structure at the vertices while considering a ``simple'' behavior on the edges (see \cite{BeKu13} for an introduction to quantum graphs).  

Our aim in this work is to study variational problems on open books and connections with their quantum graph counterparts. 
For functions  $u\in H^1_D(\Lib)$ (we refer to Section \ref{sec:functional_setting} for the precise functional setting), we define %
the \emph{action} and the \emph{Nehari} 
functionals by
\begin{align*}
    S_\omega(u)&=\frac12\norm{\nabla u}_{L^2(\Lib)}^2+\frac\omega2\norm{ u}_{L^2(\Lib)}^2-\frac1{p+1}\norm{ u}_{L^{p+1}(\Lib)}^{p+1},\\
        I_\omega(u)&=\norm{\nabla u}_{L^2(\Lib)}^2+\omega\norm{ u}_{L^2(\Lib)}^2-\norm{ u}_{L^{p+1}(\Lib)}^{p+1}.%
\end{align*}
We consider the following variational problem:
\begin{equation}\label{eqSDefIntro}
    \tilde{s}_\omega
    = \inf\Bigl\{ S_\omega(u) : u\in H^1_D(\Lib)\setminus\{0\},\ I_\omega(u)=0 \Bigr\},
\end{equation}
that is, we minimize the action over the Nehari manifold.  
The tilde $\tilde{}$ in the notation reflects the fact that we will typically work with an equivalent formulation of the problem, denoted without the tilde (see Lemma~\ref{lemSEquiv}):
    \begin{equation*}
        s_\omega = \frac{p-1}{2(p+1)} \inf\left\{\| u \|_{L^{p+1}(\Lib)}^{p+1}, \, u \in H^1_D(\Lib) \setminus\{0\}, \, I_{\omega}(u)\leq 0 \right\}.
    \end{equation*}
Minimizers of \eqref{eqSDefIntro} will be referred to as \emph{(action) ground states}.  
Such minimization problems are classical in the study of nonlinear Schrödinger equations on $\mathbb{R}^d$ and on quantum graphs.  
Ground states correspond to standing waves of the evolution equation and are expected to play a fundamental role in the long-time dynamics.
Alternative variational approaches also exist, most notably the minimization of the Schrödinger energy under a prescribed $L^2$-norm constraint.  
We refer to \cite{DeDoGaSe23, DoSeTi23, JeLu22} for a detailed comparison between these two variational frameworks.

In the present work, we focus on action ground states.  
Note that any such minimizer solves, on the book $\Lib$, the stationary nonlinear Schrödinger equation
\begin{equation}\label{eq:snls}
    Hu + \omega u - |u|^{p-1}u = 0,
\end{equation}
where $H$ denotes the Laplacian operator on $\Lib$ (see Section~\ref{sec:functional_setting} for its precise definition).  

Our first main result establishes a general existence theory.  
Let $\omega_\Lib$ be the bottom of the spectrum of $H$, defined in \eqref{eqSpBottom}.  

\begin{theorem}\label{thm:existence}
    Let $\Lib$ be a connected book, either finite or periodic, and let $\omega > -\omega_\Lib$.  
    If $\Lib$ is finite, assume additionally that $s_{\omega}<s_{\omega}^\infty$, where the action level at infinity $s_{\omega}^\infty$ is defined in \eqref{eqSOmInfinit}.  
    Then $s_\omega>0$ and there exists an action ground state, i.e., a nontrivial minimizer for $s_\omega$.
\end{theorem}

As mentioned above, minimizing the action on the Nehari manifold is a classical method, going back to \cite{Ne60}, for constructing solutions to \eqref{eq:snls}.  
In our setting, the main difficulty lies in the lack of translation invariance whenever the book is neither compact nor periodic.  
Because books may exhibit highly general geometries, one cannot directly apply the classical concentration–compactness principle of \cite{Li85}, which is typically used to recover compactness of minimizing sequences.  
Instead, one encounters a specific loss of compactness known as \emph{runaway behavior}, first identified in the context of quantum graphs in \cite{AdCaFiNo14}.  
This phenomenon motivates the additional condition $s_\omega < s_{\omega}^\infty$, in analogy with the corresponding requirement in the graph setting \cite{DeDoGaSe23, CoDoGaSeTi23}.  

Another widely used approach to obtaining solutions of \eqref{eq:snls} is the minimization of the energy under a fixed $L^2$-norm constraint (see, for example, \cite{LeSh24}).  
In contrast, our analysis focuses on action ground states.  
One reason for this choice is that action minimizers exist for every $p>1$, whereas the existence of energy ground states is usually restricted to the subcritical regime $p\in(1,3]$ \cite{BeCa81, CaLi82}.

We now present our second main result. We begin with the following definition.

\begin{definition}\label{defGraphBase}
    An open book $\Lib$ is said to be \emph{graph-based} if there exists a connected graph $\Graph$ and $L>0$ such that $\Lib$ is isomorphic to the product $\Graph \times [0,L]$. In this case, we write $\Lib_L = \Graph \times [0,L]$.
\end{definition}

We are interested in the limiting behavior of graph-based books as the transverse thickness $L$ tends to zero.  
In this regime, the natural limiting structure is the graph $\Graph$.  
For $\Lib_L = \Graph \times [0,L]$, the minimization problem \eqref{eqSDefIntro} depends on $L$, and we denote the corresponding infimum by $\tilde{s}_{\omega,L}$. As before, we work with an equivalent formulation without the tilde, see \eqref{eqSOmegaLMin}.  

Our second main result shows that the dimensionality of action ground states undergoes a sharp transition as $L$ varies.

\begin{theorem}\label{thmShrinking1}
    Let $\Graph$ be a finite or periodic graph. Let $\Lib_L = \Graph \times [0,L]$ be graph-based, and let $\omega > -\omega_{\mathcal{G}}$. Let $s_{\omega,L}$ be the minimization problem defined in \eqref{eqSOmegaLMin}. 
    Then:
\begin{enumerate}
    \item The map $L \mapsto s_{\omega,L}$ is continuous on $[0,\infty)$. Moreover, there exists $L_{min} \geq 0$ such that $s_{\omega,L}$ is constant on $[0,L_{min}]$ and strictly decreasing on $(L_{min},\infty)$.
    \item Let $s_{\omega,\Graph}$ and $s_{\omega,\Graph}^\infty$ be defined in \eqref{eqSOmInfinit} and \eqref{eqSOmega_G_Min}. If $s_{\omega,\Graph} < s_{\omega,\Graph}^\infty$, then $L_{min} > 0$. Moreover, for any $L \in [0,L_{min}]$, minimizers of $s_{\omega,L}$ exist and every minimizer $u_L$ satisfies $$\partial_y u_L \equiv 0.$$
    \item If, for some $L > L_{min}$, $s_{\omega,L}$ admits a minimizer $u_L$, then $\partial_y u_L \not\equiv 0$.
\end{enumerate}
\end{theorem}

We now comment on the above result and compare it with the existing literature.

Point~$(2)$ provides a rigorous justification for the use of quantum graphs as effective one-dimensional models for thin, two-dimensional network-shaped structures.  
While the correspondence between graphs and higher-dimensional domains is well understood in the linear setting (see, e.g., \cite{BeKu13,Ex08,Po12}), rigorous nonlinear counterparts remain scarce.  
Notable exceptions include \cite{Ko00,Ko02}, which treat compact domains and general solutions to \eqref{eq:snls}, and \cite{LeSh24}, which derives a line with a delta potential as the limiting object associated with a fractured strip in the shrinking limit for energy minimizers.

The condition $s_{\omega,\Graph} < s_{\omega,\Graph}^\infty$ guarantees the existence of ground states on the graph $\Graph$, see \cite{DeDoGaSe23,CoDoGaSeTi23}.  
In the present work, we prove that this condition is sufficient not only to ensure the existence of ground states on the book $\Lib_L$, but also to show that these ground states coincide with the graph ground states extended trivially in the transverse variable for $L\leq L_{min}$.  
Furthermore, point~$(3)$ shows that $L_{min}$ is a sharp threshold: when $L > L_{min}$, one-dimensional solitons cease to minimize the action, and genuinely two-dimensional ground states emerge.  
This behavior parallels the transverse stability/instability phenomena of line solitons known for strips of the form $\R \times \T$; see, for example, \cite{AkBaIbKi24,BeWe10,RoTz09,Ya14,Ya15}.  
Our setting includes books isomorphic to strips $\R \times [0,L]$ with either Neumann or periodic boundary conditions, allowing us to recover the dimensional transition for ground states in a unified manner.

It is also worth pointing out that in \cite{TeTzVi14}, the case of a product space $\R^d \times \mathcal M$ with $\mathcal M$ compact is studied in the context of energy ground states.  
There it is shown that, for sufficiently small $L^2$-norm, energy minimizers depend trivially on the compact variable, relying on a scaling property of the ground state in~$\R^d$.  

Similarly, in \cite{LeSh24}, a fractured strip $\R \times [0,L]$ is considered, and it is proven that energy ground states remain independent of the transverse variable for sufficiently small~$L$.  
A crucial ingredient in that analysis is the existence of an explicit and unique positive ground state on $\R$.

Our approach differs from both works in several ways. We study \emph{action} rather than energy ground states; no explicit or unique minimizer is available on graphs; and the scaling acts solely on the transverse variable (for non-uniqueness, see the recent works \cite{Do25,DoSeTe25}).  
To the best of our knowledge, action ground states have not previously been examined from this standpoint.  
This novelty, combined with the broader scope of applicability, forms a key motivation for our focus on action minimizers.  
We believe that the methods developed here can be extended to other classes of product spaces, beyond the setting of open books.

The rest of the paper is organized as follows. In Section \ref{sec:preliminaries}, we begin with a collection of preliminaries. We start in Section \ref{sec:functional_setting} by describing the precise functional setting in which we are going to work. The key point is the definition of Sobolev spaces on books, along with the description of the matching conditions at the bindings. Some notation is collected in Section \ref{sec:notation}. In Section \ref{sec:metric}, we define a metric structure on the book by constructing a suitable distance, and we introduce the concepts of connected, finite and compact books. Section \ref{sec:decay} is devoted to the proof that critical points of the action functional on books are exponentially decaying on semi-infinite pages. Section \ref{sec:examples} presents several relevant examples. 

Section \ref{secExistence} is devoted to the question of the existence of an action minimizer. We begin by reformulating the problem into an equivalent one, which corresponds to minimizing the $L^{p+1}$-norm over a side of the Nehari manifold, see the definition of $s_\omega$ in \eqref{eqSDefEquiv}. We then study the so-called problem at infinity (Section \ref{secInfinity}) in the case of finite books and we show that the level at infinity is the same as the level of the widest semi-infinite strip. Existence of an action ground state for finite (Section \ref{sec:finite-books}) and periodic (Section \ref{sec:periodic}) books is then established. For finite books, the escaping at infinity of minimizing sequences is avoided by assuming that the Nehari level $s_\omega$ is below the level at infinity $s_\omega^\infty$. For periodic books, we use in a key manner the monotonicity properties of the function $\omega\to s_\omega$ to establish the convergence of minimizing sequences. 

In Section \ref{secGraphBased}, we study the shrinking limit of graph-based books of the type $\GraphTimesZeroL$ when the length $L$ tends to $0$. We first introduce a rescaling of the problem (Section \ref{sec:rescaling}), converting the book $\GraphTimesZeroL$ into the book $\GraphTimesZeroOne$ and transferring the dependency in $L$ to the Nehari functionals. We then study the rescaled minimization level function $L\to s_{\omega, L}$: we prove that the function is continuous, constant on an interval $[0, L_{min}]$ (with possibly $L_{min}=0$), then strictly decreasing towards $0$. The rigidity of minimizers at small length is then established in Section \ref{sec:rigidity}, where the properties of the levels function are combined with the properties of the minimizers' equations to show that the dependency in the transverse variable is necessarily trivial when $L$ is small. 

\section{Preliminaries}
\label{sec:preliminaries}

\subsection{Functional setting}
\label{sec:functional_setting}

Given an open book $\Lib =\{\mathcal{B},\mathcal{P}\}$, a function $u:\Lib\to\mathbb C$  is a collection of functions $u_k:P_k\to\mathbb C$ on each of the pages $P_k\in\mathcal{P}$. 

As in the case of quantum graphs (see e.g. \cite{BeKu13}), we define the Lebesgue spaces for $p\in[1,\infty]$ and Sobolev spaces for $s\geq 0$ on the open book $\Lib$ by
\[
L^p(\Lib):=\bigoplus_{P\in\mathcal{P}} L^p(P),\quad 
H^s(\Lib):=\bigoplus_{P\in\mathcal{P}}  H^s(P). 
\]
Here no compatibility condition is imposed on the bindings, i.e., functions on the open book might be multi-valued at the bindings. From their definition, Sobolev spaces on books inherit most of the properties of Sobolev spaces on individual pages (Sobolev continuous and compact injections, Gagliardo-Nirenberg inequalities, etc.). 
For instance, if the book $\Lib$ is either finite or periodic, then for any $2\leq q<\infty$ there exists $C>0$ such that for any $u\in H^1(\Lib)$ we have
\[
\norm{u}_{H^1(\Lib)}\leq C\norm{u}_{L^q(\Lib)}.
\]

The pages of $\Lib$ are rectangles and therefore contain corners.  
While Sobolev spaces are typically introduced for smooth domains, they have also been extensively studied on polygonal domains; see, in particular, the reference monograph \cite{Gr11}.  
We recall here the results that will be used throughout the sequel.

Let $u\in H^s(\Lib)$ with $s>1/2$.  
By the trace theorem on polygonal domains in $\mathbb{R}^2$ (see \cite[Theorem~1.5.2.3]{Gr11}), one may define traces of $u$ on the bindings of the pages.  
More precisely, let $u_k : P_k \to \mathbb{C}$ denote the restriction of $u$ to the page $P_k$, and let $\{B_j\}_{j=0,\dots,J}$ denote the boundary edges (bindings) of $P_k$.  
If $u_k\in W^{s,p}(P_k)$ for some $p>1$ and $s>1/p$, then the trace operator 
$$ u_k \to \ \{u_{kj} := u_{k|B_j}\}_{j=0\dots,J} $$ 
is well defined and continuous from $W^{s,p}(P_k)$ into the product space $\prod_{j=0}^J W^{s-\frac1p,p}(B_j)$.

We now describe the compatibility conditions between traces at the corners.

When $s=1$, additional conditions arise depending on the value of $p$. Let $j,l\in\{0,\dots,J\}$, and assume that the edges $B_j$ and $B_l$ meet at a corner.  
Let
\[
\mathfrak v \in B_j \cap B_l
\]
be this corner point, which we call a \emph{vertex}.  
For $\sigma>0$ sufficiently small, let
\[
\mathfrak v - \sigma \in B_j, \qquad \mathfrak v + \sigma \in B_l
\]
denote the points obtained by moving a distance $\sigma$ away from the vertex along $B_j$ and $B_l$, respectively.  
Then the following conditions hold:

\begin{equation*}
\begin{aligned}
    u_{kj}(\mathfrak v) = u_{kl}(\mathfrak v) 
    &\qquad \text{when } p>2, \\[0.4em]
    \int_0^{\delta} \frac{1}{\sigma}\, 
        \bigl|u_{kj}(\mathfrak v - \sigma) - u_{kl}(\mathfrak v + \sigma)\bigr|^2 \, d\sigma < \infty
    &\qquad \text{when } p=2.
\end{aligned}
\end{equation*}

No compatibility condition is required when $1 < p < 2$.  
The difficulty at the endpoint $p=2$ comes from the fact that $W^{1-1/p,p}(B_j)=W^{1/2,2}(B_j)$ is exactly the threshold at which functions may fail to possess a continuous representative.

For $u\in H^2(P_k)$ and $v\in H^1(P_k)$, we have the following (half)-Green's formula on the page $P_k$ (see \cite[Lemma 1.5.3.8]{Gr11}):
\[
\int_{P_k}(\Delta u) vdx=-\int_{P_k}\nabla u\cdot\nabla v dx+\sum_{j=0,\dots,J}\int_{B_j} \eval{\frac{\partial u}{\partial \nu_j}}_{B_j}v_{|B_j}d\sigma.
\]

To analyze variational problems on the open books, we should specify how our pages are connected, i.e., we specify compatibility conditions for the functions at the bindings. It is natural to require functions to coincide on the bindings. We will be working with $H^1(\Lib)$ functions, which are not continuous, nor even pointwise defined. The fact that they coincide at the bindings is understood in the sense of traces. We introduce the notation $H^1_D(\Lib)$ (where $D$ stands for Dirichlet) for the set of $H^1(\Lib)$ functions which coincides at the bindings, i.e.
\[
H^1_D(\Lib)=
\left\{
u\in H^1(\Lib):{u_P}_{|B}={u_{P'}}_{|B},\text{ for all }B\in\mathcal{B},\text{ for all }P,P'\sim B
\right\}.
\]
We define the quadratic form $Q:L^2(\Lib)\to \mathbb R$ with domain $H^1_D(\Lib)$ by 
\[
Q(u)=\int_{\Lib}|\nabla u|^2dx=\sum_{k\in K}\int_{P_k}|\nabla u_k|^2dx.
\]
The quadratic form $Q$ is non-negative and there exists a unique self-adjoint operator $H:D(H)\subset L^2(\Lib)\to L^2(\Lib)$ such that for any $u\in D(H)$ and $v\in H^1_D(\Lib)$ we have 
\begin{equation}
\label{eq:Huv=Quv}    
\scalar{Hu}{v}_{L^2(\Lib)}=Q(u,v),
\end{equation}
where, by abuse of notation, we have also denoted by $Q$ the associated bilinear form ($Q(u,v)=Q(u+v)/2-Q(u)-Q(v)$). By definition of $H$, we have 
\[
Hu_k=-\Delta u_k,
\]
for $u=(u_k)_{k\in K}\in D(H)$, hence $D(H)\subset H^2(\Lib)\cap H^1_D(\Lib)$. Here, we have implicitly used the fact that $\{u\in H^1_D(\Lib):\Delta u\in L^2(\Lib)\}=H^2(\Lib)\cap H^1_D(\Lib)$, see \cite[Chapter 3]{Gr11} for the case of a regular domain, and \cite[Remark 3.2.4.6]{Gr11} for the case of a domain with polygonal boundary. Moreover, functions in $D(H)$ should satisfy the following binding conditions. Let $B$ be a binding and $(P_k)$ be the pages incident to the binding $B$. Any function $u\in D(H)$ verifies for any $x\in B$ the condition
\begin{equation}
\label{eq:binding-condition}    
\sum_k \eval{\frac{\partial u_k}{\partial \nu_k}}_{B}(x)=0. 
\end{equation}
Indeed, let $v\in H^1_D(\Lib)$ be supported on the pages $(P_k)$ incident to $B$ and let 
 $u\in D(H)$. By Green's formula, we have 
\begin{multline*}  
\scalar{Hu}{v}_{L^2(\Lib)}=-\Re\int_{\Lib}\Delta u \bar v dx
=-\sum_k\Re\int_{P_k}\Delta u_k \bar v_k dx
\\
=\sum_k\Re\int_{P_k}\nabla u_k\nabla \bar v_kdx-\Re\int_B\sum_k\eval{\frac{\partial u_k}{\partial \nu_k}}_{B}\bar v_{|B} d\sigma.
\end{multline*}
Since $H$ should verify \eqref{eq:Huv=Quv} and $v$ is arbitrary, this implies \eqref{eq:binding-condition}.

In summary, the domain of $H$ is given by
\begin{equation}
\label{eq:domain}    
D(H)=\left\{u\in H^2(\Lib)\cap H^1_D(\Lib):u\text{ verifies \eqref{eq:binding-condition} for each binding }B\in\mathcal{B} \right\}.
\end{equation}
Observe that the binding conditions verified by functions on the domain of $H$ are reminiscent of Kirchhoff-Neumann conditions in the context of quantum graphs. The results presented in the present paper could be generalized to more generic functionals, or, equivalently, more generic binding conditions. For example, one may introduce a Dirac-type condition, as in the case of the fractured strip studied in \cite{LeSh24}.

\subsection{Notation}
\label{sec:notation}

For a book $\Lib=(\mathcal{P},\mathcal{B})$, we define the following lower bound on the binding lengths:  
\begin{equation} \label{eqMinLength}
    L_{\Lib} = \frac{1}{2}  \min\left(\min_{B\in\mathcal{B}} l_B, 1\right),
\end{equation}
where by $l_B$ we denote the length of the binding $B$. This quantity is well defined, and we always have $L_{\Lib}\leq 1/2$. If the number of bindings is finite, then $L_{\Lib} >0$. 

We denote by $\omega_\Lib$ the infimum of the spectrum of $H$, which is given by 
\begin{equation}
    \label{eqSpBottom}
    \omega_{\Lib} = \inf_{u \in D(Q)} \frac{Q(u)}{\| u \|_{L^2(\Lib)}^2}. 
\end{equation}
Since we assumed Kirchhoff-type conditions at the bindings, we have $\omega_\Lib = 0$. In the sequel, we chose to keep the notation $\omega_\Lib$, as most of our statements would be valid for operators with more generic boundary conditions (and thus potentially non-zero $\omega_\Lib$). 

For any $u \in H^1_D(\Lib)\setminus\{0\}$ and $\omega> -\omega_\Lib$, we define
    \begin{equation}\label{eqPiOmega}
        \pi_{\omega}(u) =  \left( 1 + \frac{I_{\omega}(u)}{\| u \|_{L^{p+1}(\Lib)}^{p+1}} \right)^{\frac{1}{p-1}}.
    \end{equation}
It is the scaling factor used to shift $u$ on the Nehari manifold, i.e., $I_\omega(\pi_{\omega}(u) u) = 0$.

\subsection{Metric structure}
\label{sec:metric}
Open books can be endowed with a metric structure.  

The distance between two points lying within the same page is simply the Euclidean distance inherited from the page. This includes points lying on a binding, whose distance with respect to any point belonging to the pages incident to the binding is therefore defined. This is sufficient to define a notion of continuity of a curve. To define the distance between two points belonging to different pages not sharing a common binding, we proceed in the following way.

Let $\Lib$ be a book. Let $x,y\in\mathcal{L}$. 
A path $\gamma$ on $\mathcal{L}$ between $x$ and $y$ is a continuous piecewise  $C^1$ application $\gamma:[0,1]\to\mathcal{L}$ such that $\gamma(0)=x$ and $\gamma(1)=y$. 
The length of the path on $\mathcal{L}$ is the sum of the lengths of $\gamma([0,1])$ restricted to each of the pages of $\Lib$ (counting the length on common bindings only once). 
We denote it by $\ell_\gamma$. The distance between $x$ and $y$ on $\mathcal{L}$ is then defined by
\[
d(x,y)=\inf\left\{\ell_\gamma:\gamma\in PC^1([0,1];\Lib),\,\gamma(0)=x,\,\gamma(1)=y\right\}.
\]

A book $\mathcal{L}$ is said to be \emph{connected} if each binding has at least one incident page, and if for any two pages $P, P'\in\mathcal{P}$, there exist a sequence of pages $(P_j)_{j=0,\dots, J+1}$ and of bindings $(B_j)_{j=0,\dots, J}$ such that $P_0=P$, $P_{J+1}=P'$ and $P_j, P_{j+1}\sim B_j$. 
In this paper, all books will be assumed to be connected. 

We say that a book is \emph{finite} if it has a finite number of pages and bindings. 

We say that a book is \emph{compact} if it is finite and each of the bindings has a finite length. Otherwise, we say that the book is \emph{non-compact}. When a book is finite but non-compact, we define its \emph{compact core} as the sub-book built with the collection of pages such that the associated bindings all have a finite length. 

\subsection{Exponential decay}
\label{sec:decay}

In this section, we show that the solutions of \eqref{eq:snls} are exponentially decaying. 

One of the main references for exponential decay in elliptic linear equations is by Agmon \cite{Ag82}.
In our case, we do not really need to have a precise estimate on the decay rate of solutions to \eqref{eq:snls}, and we can settle for a slightly weaker estimate (still giving an exponential decay rate, though not the optimal one). We follow the strategy of proof of \cite[Theorem 3.2]{BeSh91}. 

\begin{proposition}
    Let $\Lib$ be a book and let $u\in H^1_D(\Lib)$ be a solution of \eqref{eq:snls}. For any page $P_k\in\mathcal{P}$ such that $L_k^1=\infty$ (resp. $L_k^2=\infty$), there exists $M_k>0$ such that for any $(x,y)\in[0,L_k^1]\times [0,L_k^2]$ the component $u_k$ of $u$ on $P_k$ verifies 
    \[
    |u_k(x,y)|\leq M_ke^{-\frac{\sqrt{\omega}}{2}|x|} \quad (\text{resp. } |u_k(x,y)|\leq M_ke^{-\frac{\sqrt{\omega}}{2}|y|}).
    \]
    \end{proposition}

\begin{proof}
Let $u\in H^1_D(\Lib)$ be a solution of \eqref{eq:snls}. Consider a page $P_k\sim[0,L_k^1]\times [0,L_k^2]$ of $\Lib$ and assume that $L_k^1=\infty$ (the case $L_k^2=\infty$ being perfectly similar). Let $u_k:P_k\to\R$ be the component of $u$ on $P_k$.  
Define $\psi:[0,\infty)\times [0,L_k^2]\to \R$ by
\[
\psi=|u_k|^2-Mg_0,\quad g_0(x,y)=e^{-\sqrt{\omega}x},
\]
where $M>0$ is a constant to be chosen and $g_0$ has been chosen so that on $[0,\infty)\times [0,L_k^2]$ it verifies 
\[
-\Delta g_0+\omega g_0=0.
\] 
Let $R>0$ to be chosen large enough later, and fix $M$ large enough so that when $x=R$ we have 
\[
\psi(R,y)<0.
\]
We will prove that in fact, for any $x>R$ and $y\in[0,L_k^2]$, we have
\[
\psi(x,y)<0,
\]
thereby proving the claim. On the strip $(0,\infty)\times (0,L_k^2)$, $u_k$ verifies
\[
-\Delta u_k+\omega u_k-|u_k|^{p-1}u_k=0.
\]
Therefore, $u_k\in\mathcal{C}^2((0,\infty)\times (0,L_k^2))\cap H^2((0,\infty)\times (0,L_k^2))$ and  $|u_k|^2$ verifies 
\[
-\Delta |u_k|^2 = -2 \Re\left(u_k\Delta \bar u_k+|\nabla u_k|^2\right)
= -2 \left(\omega |u_k|^2-|u_k|^{p+1}+|\nabla u_k|^2\right).
\]
We rewrite this equation in the form 
\begin{equation}
\label{eq:take-R-large}    
-\Delta |u_k|^2 +\omega |u_k|^2
= (-\omega+2|u_k|^{p-1}) |u_k|^2-2|\nabla u_k|^2,
\end{equation}
in such a way that the right-hand side is negative for small $|u_k|$. Since $\lim_{x\to \infty}u_k(x,y)=0$, we may choose $R$ large enough so that the right hand side of \eqref{eq:take-R-large} is negative on $[R,\infty)\times [0,L_k^2]$. 

By construction, the function $\psi$ also verifies
\[
-\Delta \psi +\omega \psi
= (-\omega+2|u_k|^{p-1}) |u_k|^2-2|\nabla u_k|^2.
\]
Recall that, from the maximum principle, if $-\Delta\psi+\omega\psi \leq 0$ on a domain $\Omega$, then $\psi$ cannot have a positive maximum in $\Omega$. Let $\rho>R$ and $\Omega =[R,\rho]\times [0,L_k^2]$. 
By the maximum principle, $\psi$ can achieve a positive maximum only on $\partial \Omega$. By construction, it cannot be on the part $\{R\}\times [0, L_k^2]$ of $\partial \Omega$. Moreover, it also cannot be on $[R,\rho]\times\{0 ,L_k^2\}$. Indeed, assume by contradiction that $\psi$ achieves a positive maximum at $(x,0)$ for $x\in(R,\rho)$ and consider the symmetrized function $\tilde \psi$ given by $\psi(x,y)$ for  $(x,y)\in\Omega$ and $\psi(x,-y)$ for $(x,y)\in[R,\rho]\times[-L_k^2,0]$. Then $\tilde\psi$ 
also verifies $-\Delta\psi+\omega\psi\leq 0$ on $[R,\rho]\times[-L_k^2,L_k^2]$ and achieves a positive maximum at the interior point $(x,0)$, which is a contradiction. Therefore, $\psi$ can achieve a positive maximum only on $\{\rho\}\times [0,L_k^2]$.
Define $m(\rho)$ by 
\[
m(\rho)=\max_{y\in [0,L_k^2]}|\psi(\rho,y)|.
\]
Then any $(x,y)\in \Omega$ we have 
\[
\psi(x,y)\leq m(\rho).
\]
As $\rho\to \infty$, we have $m(\rho)\to 0$, therefore on $[R,\infty)\times [0,L_k^2]$ we have
\[
\psi(x,y)\leq 0.
\]
As a consequence, for any $(x,y)\in[R,\infty)\times [0,L_k^2]$ we have 
\[
|u_k(x,y)|^2\leq Mg_0(x) = Me^{-\sqrt{\omega}|x|},
\]
which is the desired result.
\end{proof}

\subsection{Examples}
\label{sec:examples}

We now present several examples of open books.
Note that while we might visually represent the open-book structures as submanifolds of the space $\mathbb R^3$,  the geometry of the representation is not taken into account in the open-book object: a curved binding or page is identical to a straight one (in the same way that in the representation of metric graphs, one can use either curved or straight edges for the sake of visualization to represent the same structure). 

The notion of open books allows for a wide variety of constructions. Our primary motivation for the introduction of this notion was to study how quantum graphs, which are $1$-$d$ structures, could be used as approximations of more complicated $2$-$d$ structures. For any given quantum graph, it turns out to be elementary to consider its open book equivalent, simply by giving a dimension to the vertices and edges, i.e. by considering the product space $\GraphTimesZeroL$. More precisely, given a graph $\Graph$ formed by edges $e\in\mathcal{E}$ of length $l_e$ and vertices $v\in\mathcal{V}$, we construct an open book $\Lib$ as follows. Let $L>0$. Given an edge $e\sim [0,l_e]$  attached at two vertices $v_1$ and $v_2$ (at respectively $0$ and $l_e$), we define a page $P_e$ as $P_e\sim [0,l_e]\times [0,L]$ and binding $(B_j)_{j=0,1,2,3}$ as $B_1,B_3\sim [0,L]$, $B_0,B_2\sim [0,l_e]$, thereby forming the boundary of $P_e$ in such a way that the vertices $v_1$ and $v_2$ become respectively the bindings $B_3$, $B_1$. This process can be extended, mutatis mutandis, for edges of infinite length or loop edges. Obviously, there would be many other ways to construct an open book starting from a given quantum graph.

Our first example is the one whose graphical representation justifies the name \emph{open books}, and can be thought of as a generalized star graph. We give ourselves a binding $B_0 \sim[0,L]$, and for $k=0,\dots,K$ we define bindings $B_1^{k},B_2^{k}  \sim[0,\infty)$. Then the $K$ pages attached to the common binding $B_0$ are isomorphic to $[0,L]\times [0,\infty)$ and their boundary are formed with the bindings in the  following way
\[
\partial P_k = B_1^k\cup B_0 \cup B_2^k.
\]
An open book with three pages is represented in Figure \ref{fig:classical-graphs} (left).
\begin{figure}[htpb!]
    \centering
\includegraphics[width=0.7\linewidth]{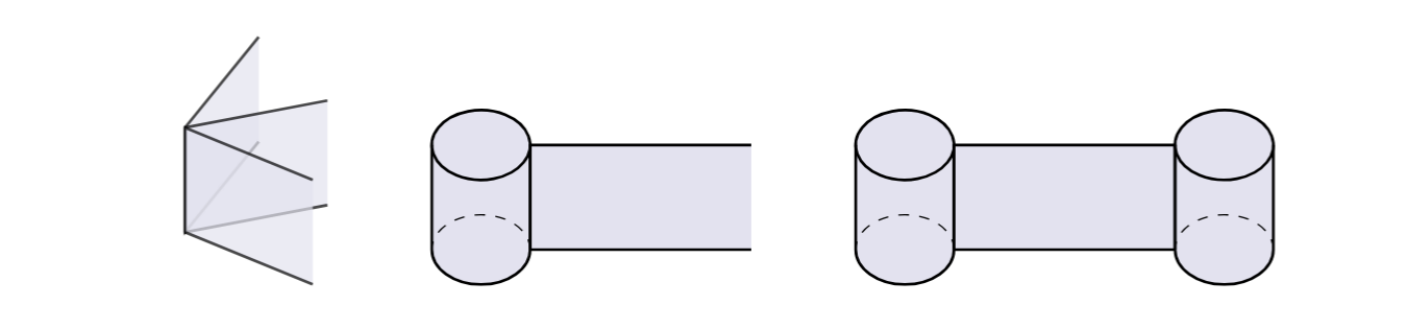}
    \caption{Open book versions of classical graphs: star-graph, tadpole and dumbbell (from left to right)}
    \label{fig:classical-graphs}
\end{figure}

Following the same procedure, we construct a generalized tadpole. Take bindings $B_0  \sim[0,L]$, $B_1^{\pm} = [0,\infty)$, $B_2^{\pm} = [0,2\pi]$, and let  two pages $P_1$ and $P_2$ be such that $P_1$ is isomorphic to $[0,L]\times [0,\infty)$ and $P_2$ is isomorphic to  $[0,L]\times [0,2\pi]$, and their boundaries are described in the following way. For $P_1$ we have 
\[
[0,L]\times\{0\} \sim B_0,
\quad
\{0\}\times [0,\infty) \sim B_1^-,
\quad
\{L\}\times [0,\infty) \sim B_1^+
\]
and for $P_2$ we have
\[
[0,L]\times\{0\} \sim [0,L]\times\{2\pi\}\sim B_0,\quad \{0\}\times[0,2\pi]\sim B_2^-,\quad \{L\}\times[0,2\pi]\sim B_2^+.
\]
A generalized tapdole is represented in Figure \ref{fig:classical-graphs} (middle). 

The procedure can be repeated for any example of a quantum graph. 
We have included in Figure \ref{fig:classical-graphs} (right) the open-book version of the dumbbell graph. Figure \ref{fig:grid-torus} presents the open book version of the grid. This example served as a base for the quantum graph approach to metamaterial design presented in \cite{LaTaCh22}.

\begin{figure}
    \centering
\includegraphics[width=0.7\linewidth]{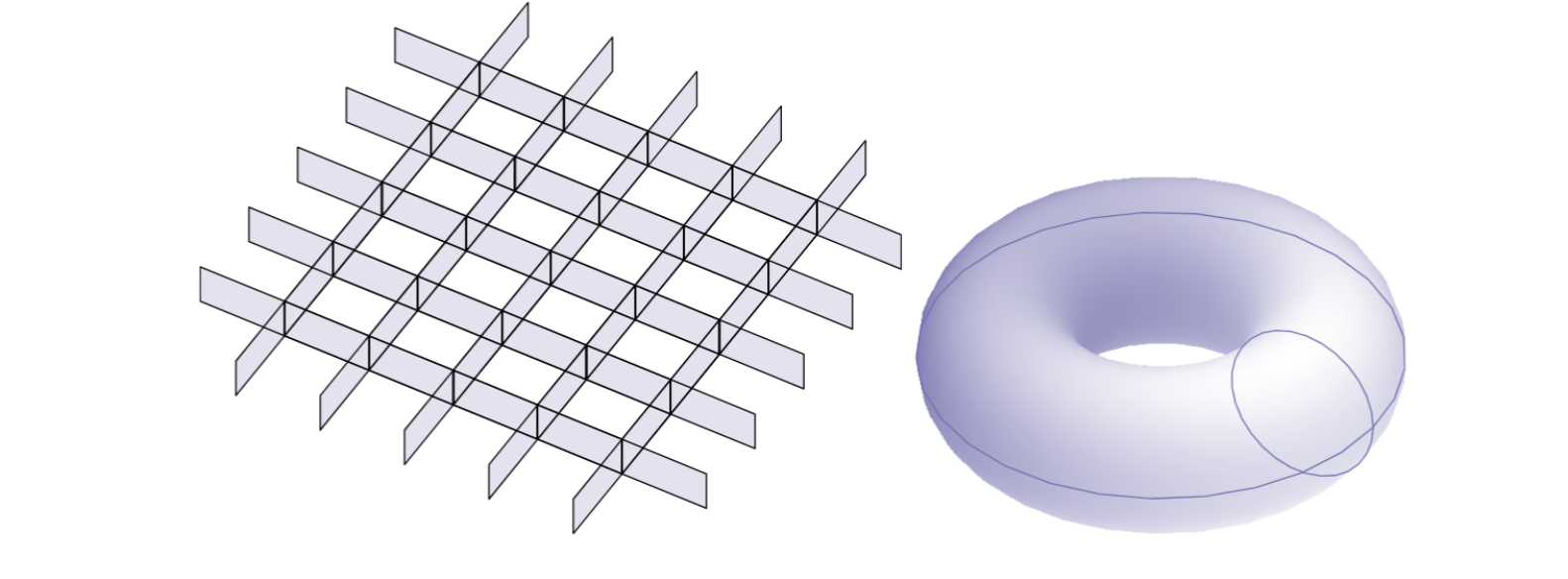}
    \caption{The open-book versions of the 2-d grid and the torus}
    \label{fig:grid-torus}
\end{figure}

Not all open books can be thought of as extensions of metric graphs. For example, there is no natural way to obtain the open book of Figure \ref{fig:proto} from a graph. 

We present a last example of an open book: the torus. It is constructed from a single page $P_1$ for which the boundary bindings are two by two identical, i.e., $P_1^N = P_1^S = B_0$ and $P_1^W = P_1^E = B_1$. The torus open book with solid lines for the bindings is represented on Figure \ref{fig:grid-torus}.

\section{Existence of an action minimizer}\label{secExistence}

In this section, we prove the  existence of an action minimizer on finite non-compact books and on periodic  books. We start by reformulating the minimization problem into an equivalent problem more amenable to analysis.

\begin{lemma}\label{lemSEquiv}
Let $\Lib$ be a book.
Assume that $\omega> -\omega_\Lib$. The minimization problem \eqref{eqSDefIntro} is equivalent to 
    \begin{equation}
        \label{eqSDefEquiv}
        s_\omega = c_p \inf\left\{\| u \|_{L^{p+1}(\Lib)}^{p+1}, \, u \in H^1_D(\Lib) \setminus\{0\}, \, I_{\omega}(u)\leq 0 \right\},
    \end{equation}
    where 
    \begin{equation}
        \label{eqDefCp}
        c_p = \frac{p-1}{2(p+1)}.
    \end{equation}
\end{lemma}
\begin{proof}
    Notice that 
    \begin{equation*}
        c_p\| u \|_{L^{p+1}(\Lib)}^{p+1} = S_\omega(u) - \frac12I_\omega(u),
    \end{equation*}
    therefore problem \eqref{eqSDefIntro} is equivalent to 
    \begin{equation*}        
         \hat{s}_\omega = c_p \inf\left\{\| u \|_{L^{p+1}(\Lib)}^{p+1}, \, u \in H^1_D(\Lib) \setminus\{0\}, \, I_{\omega}(u) = 0 \right\}.
    \end{equation*}
        On the one hand, we clearly have $s_\omega \leq \hat{s}_\omega $. On the other hand, suppose that $u\in H^1_D(\Lib)\setminus\{0\}$ verifies  $I_\omega(u) < 0$. By definition of $\pi_{\omega}(u)$ (see \eqref{eqPiOmega}), we have 
        \[
        I_\omega(\pi_{\omega}(u) u)  = 0
        \]
        while, since $I_\omega( u)<0$, we have $\pi_{\omega}(u)<1$ and therefore 
        \[
        \| \pi_{\omega}(u) u \|_{L^{p+1}(\Lib)}^{p+1} < \| u \|_{L^{p+1}(\Lib)}^{p+1}.
        \]
        This implies that $s_\omega \geq \hat{s}_\omega $. Since the reverse inequality is also true, this implies that $s_\omega =\hat{s}_\omega $ and therefore, as stated, $s_\omega$ is equivalent to $\tilde{s}_\omega $.
\end{proof}

We then show that $\omega \geq - \omega_\Lib$ is a necessary condition for the existence of non-trivial action ground states.

\begin{lemma}
\label{lem:omega_less_omega_L}
Let $\Lib$ be a book.
    If $\omega < - \omega_\Lib$ then $s_\omega=0$ and a non-trivial minimizer to \eqref{eqSDefEquiv} does not exist.
\end{lemma}

\begin{proof}
    When $\omega \leq - \omega_\Lib$, the Nehari manifold is not bounded away from $0$; it is the key point that we are going to exploit. Assume that $\omega < - \omega_\Lib$. Then there exists $u \in H^1_D(\Lib)\setminus\{0\}$  such that 
     \[
    \norm{\nabla u}_{L^2(\Lib)}^2+\omega\norm{u}_{L^2(\Lib)}^2=\left(\frac{\norm{\nabla u}_{L^2(\Lib)}^2}{\norm{u}_{L^2(\Lib)}^2}-\omega_\Lib+(\omega+\omega_\Lib)\right)\norm{u}_{L^2(\Lib)}^2\leq 0.
    \]
Let $(\lambda_n)\in(0,\infty)$ be such that $\lambda_n\to 0$, and define $(u_n)\subset H^1_D(\Lib)\setminus\{0\}$ by $u_n=\lambda_nu$. Then $I_\omega(u_n)<0$ and $c_p\norm{u_n}_{L^{p+1}(\Lib)}^{p+1}\to 0$ as $n\to\infty$. This implies that $s_\omega=0$. 
\end{proof}

The next two Lemmas are used in Section \ref{sec:periodic} for periodic books but apply generically.

\begin{lemma}\label{lemSOmMonotone}
Let $\Lib$ be a book.
Assume that $\Lib$ is either finite or periodic. Then
   the function $\omega \mapsto s_\omega$ is $0$ for $\omega\in (-\infty,\omega_\Lib)$ and is strictly increasing for $\omega \in (\omega_\Lib, \infty)$.
\end{lemma}
\begin{proof}
We already proved in Lemma \ref{lem:omega_less_omega_L} that the function $\omega \mapsto s_\omega$ is $0$ for $\omega <-\omega_\Lib$. For $\omega>\omega_\Lib$, we have $s_\omega>0$ as a consequence of Sobolev inequalities (we assumed that $\Lib$ is either finite or periodic to ensure the validity of Sobolev inequalities). Indeed, let $\omega>\omega_\Lib$ and $u\in H^1_D(\Lib)\setminus\{0\}$ such that $I_\omega(u)=0$. Then, by Sobolev embeddings, there exists $C>0$ independent of $u$ such that 
\[
\norm{\nabla u}_{L^2(\Lib)}^2+\omega\norm{u}_{L^2(\Lib)}^2=\norm{u}_{L^{p+1}(\Lib)}^{p+1}\leq C(\norm{\nabla u}_{L^2(\Lib)}^2+\omega\norm{u}_{L^2(\Lib)}^2)^{\frac{p+1}{2}}.
\]
Therefore, there exists $c>0$ independent of $u$  such that 
\[
c\leq \norm{\nabla u}_{L^2(\Lib)}^2+\omega\norm{u}_{L^2(\Lib)}^2.
\]
Therefore, functions on the Nehari manifold are uniformly bounded away from $0$ in $H^1(\Lib)$ and in $L^{p+1}(\Lib)$. As a consequence, we have $s_\omega>0$. 

    We now prove that the function $\omega \mapsto s_\omega$ is increasing. Let $\omega_1 < {\omega_2}$, $\omega_1,{\omega_2} \in (-\omega_{\Lib}, \infty)$. 
    Let $(u_n) \subset H^1_D(\mathcal{L)}$ be such that $I_{\omega_2}(u_n) = 0$ and $s_{\omega_2}\leq c_p\norm{u_n}_{L^{p+1}(\Lib)}^{p+1}<s_{\omega_2}+\frac1n$. Then $I_{\omega_1}(u_n)=-(\omega_2-\omega_1)\norm{u_n}_{L^2(\Lib)}^{2} < 0$, 
      $\pi_{\omega_1}(u_n) < 1$ (where $\pi_{\omega_1}$ is defined in \eqref{eqPiOmega})
    and 
    \[
    s_{\omega_1}\leq c_p\| \pi_{\omega_1}(u_n) u_n \|_{L^{p+1}(\Lib)}^{p+1} < c_p\|  u_n \|_{L^{p+1}(\Lib)}^{p+1}=s_{\omega_2}+\frac1n.
    \] 
    Passing to the limit as $n\to\infty$ leads to $s_{\omega_1}\leq s_{\omega_2}$. 
        
    We now prove that the function is strictly increasing. Assume by contradiction that  $s_{\omega_1}= s_{\omega_2}$. Then $\pi_{\omega_1}(u_n)\to1$ as $n\to\infty$, which, by definition of $\pi_{\omega_1}$ and since $s_{\omega_1}>0$, implies that $\norm{u_n}_{L^2(\Lib)}^{2}\to 0$ as $n\to\infty$. Let $q>p+1$. By interpolation, there exists $\alpha\in(0,1)$ such that  we have
    \[
    \norm{u_n}_{L^{p+1}(\Lib)}\leq C\norm{u_n}_{L^2(\Lib)}^{\alpha}\norm{u_n}_{L^q(\Lib)}^{1-\alpha}.
    \]
    From Sobolev embeddings, we have 
    \[
    \norm{u_n}_{L^q(\Lib)}\leq C\norm{u_n}_{H^1(\Lib)}.
    \]
    Moreover, since  $I_{\omega_2}(u_n)=0$ and $c_p\|  u_n \|_{L^{p+1}(\Lib)}^{p+1}\to s_{\omega_2}$  as $n\to\infty$, the sequence $(u_n)$ is bounded in $H^1(\Lib)$ and we have 
    \[
    0<s_{\omega_2}\leq C\norm{u_n}_{L^2(\Lib)}^{\alpha}\norm{u_n}_{H^1(\Lib)}^{1-\alpha}.
    \]
    Since $\norm{u_n}_{L^2(\Lib)}\to 0$ as $n\to\infty$, this gives a contradiction. Therefore, $s_{\omega_1}<s_{\omega_2}$, which concludes the proof. 
\end{proof}

The following lemma is the book version of a lemma often used in concentration compactness arguments (see \cite{Li84a,Li84b}).
\begin{lemma}\label{lemLions}
Let $\Lib$ be a book.
Assume that $\Lib$ is either finite or periodic. 
    Let $r \in (0, L_{\Lib}/2)$. Let $(u_n)$ be a sequence bounded in $H^1(\Lib)$. If 
    \begin{equation*}
        \sup_{y \in \Lib} \int_{B(y,r) } |u_n|^2 \, dx \to 0, \quad n\to \infty
    \end{equation*}
    then $u_n \to 0$ in $L^q(\Lib)$ for $2 < q < \infty$.
\end{lemma}

\begin{proof}
Observe first that we have defined a distance on books, and that the balls $B(y,r)$ are defined with respect to this distance.
    By the Gagliardo-Nirenberg inequality, we have 
    \begin{equation*}
        \| u \|_{L^q(B(y,r) )}^q \lesssim \| u \|_{L^2(B(y,r) )}^2 \left(1+ \|\nabla u \|_{L^2(B(y,r) )}^{q - 2}\right) 
    \end{equation*}
    for any $u\in H^1(\Lib)$,  $y \in \Lib$ and $q \in [2,\infty)$. Note that the above Gagliardo-Nirenberg is valid since we have assumed that the book $\Lib$ is either finite or periodic. It would be possible to relax this assumption and still have a valid Gagliardo-Nirenberg inequality, but for the sake of simplicity, we refrained from optimizing.
    By covering $\Lib$ with balls of radius $r$ in such a way that any point is contained in at most $N \geq 1$ balls,  we obtain 
    \begin{equation*}
          \| u \|_{L^q(\Lib)}^q \lesssim N \sup_{y \in \Lib} \| u \|_{L^2(B(y,r) )}^2 \left(1+ \| u \|_{H^1(\Lib)}^{q-2}\right).
    \end{equation*}
    Thus, under the assumptions of the lemma, we have $u_n \to 0$ in $L^{p}(\Lib)$. 
\end{proof}

\subsection{The problem at infinity} \label{secInfinity}

In many situations, the so-called \emph{problem at infinity} plays a special role in the analysis of the existence of minimizers, in connection with concentration-compactness arguments (see e.g. \cite{JeTa02,Li84b}). In the present setting, it is defined by 
\begin{equation}
    \label{eqSOmInfinit}
    s_\omega^\infty = \inf\left\{\liminf_{n \to \infty} 
    c_p\norm{u_n}_{L^{p+1}(\Lib)}^{p+1}
    : u_n \rightharpoonup 0 \text{ in } H^1(\Lib),u_n\not \equiv 0, I_\omega(u_n)\leq 0\right\}.
\end{equation}
Observe that it always holds that $s_{\omega} \leq s_{\omega}^\infty$. 

Given a graph $\Graph$, for notational convenience, we introduce  the Nehari functional on the graph:
\begin{equation*}
     I_{\omega,\Graph}(u) = \| \partial_x u \|_{L^2(\Graph)}^2  + \omega\| u \|_{L^2(\Graph)}^2 -  \|  u \|_{L^{p+1}(\Graph)}^{p+1},
\end{equation*}
and the corresponding minimization problem
\begin{equation}   
\label{eqSOmega_G_Min}
    s_{\omega,\Graph} = \inf\{c_p \| u \|_{L^{p+1}(\Graph)}^{p+1}: u \in H^1_D(\Graph)\setminus\{0\}, I_{\omega,\Graph}(u) \leq 0\}.
\end{equation}
We denote by $s_{\omega,\Graph}^\infty$ the level at infinity on the graph, defined equivalently as $s_\omega^\infty$ in \eqref{eqSOmInfinit}.

In this section, $\omega\in\R$ will be assumed to be such that $\omega>-\omega_\Lib\geq 0$ or $\omega>-\omega_\Graph\geq 0$, where $\Lib$ or $\Graph$ is the underlying book or graph.

We will link with the minimal action  levels on the line and on the strip of width $L$, defined by 
    \begin{align*}
                 s_{\omega}^{\mathrm{line}}& = c_p \inf \left\{ \| u \|_{L^{p+1}(\R)}^{p+1} : u \in H^1(\R) \setminus\{0\},I_{\omega}^{\mathrm{line}}(u) \leq 0 \right\}, \\
                s_{\omega}^{\mathrm{strip}_L} &= c_p \inf\left \{ \| u \|_{L^{p+1}(\R \times [0,L])}^{p+1} : u \in H^1(\R\times [0,L]) \setminus\{0\},I_{\omega}^{\mathrm{strip}_L}(u) \leq 0 \right\},
    \end{align*}
where, by $I_{\omega}^{\mathrm{line}}$ and $I_{\omega}^{\mathrm{strip}_L}$ we denote the Nehari functionals on the line and on the strip of width $L$, i.e, for $u\in H^1(\R)$,
\[
I_{\omega}^{\mathrm{line}}(u)=\norm{\partial_x u}_{L^2(\R)}^2+
\omega\norm{ u}_{L^2(\R)}^2-
\norm{ u}_{L^{p+1}(\R)}^{p+1},
\]
and   for $u\in H^1(\R\times [0,L])$,
\[
I_{\omega}^{\mathrm{strip}_L}(u)=\norm{\nabla u}_{L^2(\R\times [0,L])}^2+
\omega\norm{ u}_{L^2(\R\times [0,L])}^2-
\norm{ u}_{L^{p+1}(\R\times [0,L])}^{p+1}.
\]

We start by showing that the level of the problem at infinity for non-compact finite books is the same as the action level on a strip. Similar arguments also show that the level of the problem at infinity on non-compact finite graphs is the same as the action level on a line. The reason is that the best escaping sequences minimizing the action on the Nehari constraint for books reduce to the simple escaping of strip-ground states on a single page isomorphic to a half-strip (or on a semi-infinite edge in the case of graphs).

\begin{lemma}\label{lemInfty}
Let $\mathcal L=(\mathcal P,\mathcal B)$ be a non-compact finite book. Assume that for any page $P_k\in \mathcal P$, we have either $L_k^1<\infty$ or $L_k^2<\infty$. Then 
\[
s_{\omega}^\infty = s_{\omega}^{\mathrm{strip}_{L^\vee}},\quad L^{\vee}=\max\{L_k^j:L_k^{3-j}=\infty\}.
\]
Let $\Graph$ be a non-compact finite graph. Then
\[
s_{\omega,\Graph}^\infty = s_{\omega}^{\mathrm{line}}.
\]
\end{lemma}

\begin{proof}
We provide the proof for books, the proof for graphs being similar and easier. 
Let $\mathcal L=(\mathcal P,\mathcal B)$ be a non-compact finite book such that for any page $P_k\in \mathcal P$, we have either $L_k^1<\infty$ or $L_k^2<\infty$. We first observe that 
\begin{equation}
\label{eq:ref1}    
s_{\omega}^\infty\leq s_{\omega}^{\mathrm{strip}_{L^{\vee}}}.
\end{equation}
Indeed, let $(w_n)\subset H^1(\R\times[0,L^{\vee}])$ be a minimizing sequence for $s_{\omega}^{\mathrm{strip}_{L^{\vee}}}$, i.e. $w_n\neq 0$, $I_{\omega}^{\mathrm{strip}_{L^{\vee}}}(w_n)=0$  and $c_p\norm{w_n}_{L^{p+1}(\R\times [0,L^{\vee}])}^{p+1}\to s_{\omega}^{\mathrm{strip}_{L^{\vee}}}$. Let $\chi:\R\to[0,1]$ be a smooth cut-off function verifying $\chi(x)=0$ for $x\in(-\infty,0]$ and $\chi(x)=1$ for $x\in[1,\infty)$. Let $(x_n)\subset \R$ be such that $x_n\to\infty$ as $n\to\infty$ and $\tilde w_n$ defined by $\tilde w_n(x,y)=\chi(x)w_n(x-x_n,y)$. The sequence $(x_n)$ is chosen such that $\tilde w_n$ verifies 
$I_{\omega}^{\mathrm{strip}_{L^{\vee}}}(\tilde w_n)\leq 1/n$. Moreover, $\norm{\tilde w_n}_{L^{p+1}(\R\times [0,L^{\vee}])}^{p+1}\leq \norm{ w_n}_{L^{p+1}(\R\times [0,L^{\vee}])}^{p+1}$. Define $\pi_n>0$ by $I_{\omega}^{\mathrm{strip}_{L^{\vee}}}(\pi_n\tilde w_n)=0$. Then $\pi_n\to 1$ and   $(\pi_n\tilde w_n)$ is also a minimizing sequence of $s_{\omega}^{\mathrm{strip}_{L^{\vee}}}$, which moreover verifies $\pi_n\tilde w_n\rightharpoonup 0$ in $H^1([0,\infty)\times [0,L^{\vee}])$ as $n\to\infty$.
Up to renumbering, we may assume that the page $P_1\in\mathcal P$ is such that $P_1\sim[0,\infty)\times [0,L^{\vee}]$. 
Let $v_n=(v_n^k)\in H^1_D(\Lib)$  be defined by $v_n^1=\pi_n\tilde w_n$ (restricted to $P_1$) and $v_n^k\equiv 0$ for $k\geq 2$. Then $v_n\rightharpoonup 0$ weakly in $H^1(\Lib)$ as $n\to\infty$, $I_{\omega}(v_n)=0$ and 
\[
s_{\omega}^\infty\leq \lim_{n\to\infty} c_p\norm{v_n}_{L^{p+1}(\Lib)}^{p+1}=
\lim_{n\to\infty} c_p\norm{\pi_n\tilde w_n}_{L^{p+1}([0,\infty)\times[0,L^{\vee}])}^{p+1}=s_{\omega}^{\mathrm{strip}_{L^{\vee}}}.
\]

We now show that the reverse inequality also holds. 
 Let $u_n \rightharpoonup 0$ in $H^1_D(\Lib)$ be a minimizing sequence for $s_{\omega}^\infty$, that is $c_p \| u_n \|_{L^{p+1}(\Lib)}^{p+1} \to s_{\omega}^\infty$ and $I_{\omega}(u_n) = 0$. Let $\mathcal{K}$ be the compact core of $\Lib$. Let $\phi$ be a smooth cut-off function outside of $\mathcal{K}$, that is $\phi \equiv 1$ on $\mathcal{K}^c \setminus \mathcal{B}$, $\phi \equiv 0$ on $\mathcal{K}$ and $\mathcal{B}$ a compact transition region, $\phi \in [0,1]$ on $\mathcal{B}$. Let $v_n = \phi u_n$. Since $u_n \rightharpoonup 0$, it follows that $\| u_n \|_{L^{p+1}(\Lib)} = \| v_n \|_{L^{p+1}(\Lib)} + \eps_n$, where $\eps_n \to 0$ as $n\to\infty$. In particular, we have $c_p \| v_n \|_{L^{p+1}(\Lib)}^{p+1} \to s_{\omega}^\infty$ as $n\to\infty$. In the same way, we have
    \begin{equation} \label{eqNehConv1}
    I_{\omega}(v_n) = I_{\omega}(u_n) + \eps'_n
    \end{equation}
     where $\eps'_n \to 0$ as $n\to\infty$. Let $w_n = \pi_\omega(v_n) v_n$ where $\pi_\omega(v_n) v_n$ is the projection on the Nehari manifold defined in \eqref{eqPiOmega} (in particular $I_\omega(w_n) = 0$). From \eqref{eqNehConv1}, we obtain that $\pi_\omega(v_n) \to 1$,  thus  
     $$c_p \| w_n \|_{L^{p+1}(\Lib)}^{p+1} = c_p \pi_\omega(v_n)^{p+1} \| v_n \|_{L^{p+1}(\Lib)}^{p+1} \to s_{\omega}^\infty$$ 
     as $n\to\infty$. In particular, there exists $(\delta_n)\subset(0,\infty)$, with $\delta_n \to 0$ as $n\to\infty$, such that 
     \begin{equation}\label{eqConverge2}
         c_p \| w_n \|_{L^{p+1}(\Lib)}^{p+1} = s_{\omega}^\infty + \delta_n.
     \end{equation}
     Now observe that $w_n$ is a collection of $K$ disjoint pieces, where $K$ is the number of semi-infinite pages of the book $\Lib$. Indeed, by construction, $w_n$ vanishes on the compact core $\mathcal{K}$ of the book. Writing each semi-infinite page $P_k$ of $\Lib$ as $[0,\infty)\times [0,L_k]$, we have
     \begin{equation*}
         w_n = (w_n^{k}), \quad  w_n^{k} \subset H^1([0,\infty)\times [0,L_k]),\quad k =1,\dots,K.
     \end{equation*}
Since $0=I_{\omega}(w_n)=\sum_kI_{\omega}^{\mathrm{strip}_{L_k}}(w_n^k)$,
there exists $k\in\mathbb N$ such that $I_{\omega}^{\mathrm{strip}_{L_k}}(w_n^k)\leq 0$. Therefore 
\[
s_{\omega}^{\mathrm{strip}_{L_k}}\leq c_p\norm{w_n^k}_{L^{p+1}([0,\infty)\times[0,L_k])}^{p+1}\leq c_p\norm{w_n}_{L^{p+1}(\Lib)}^{p+1}.
\]
Passing to the limit, we obtain
\[
s_{\omega}^{\mathrm{strip}_{L_k}}\leq s_\omega^\infty.
\]
As $L\to s_{\omega}^{\mathrm{strip}_L}$ is decreasing in $L$ (this is proved in a general case in Proposition \ref{propMonoton}), this implies that  
\begin{equation}
\label{eq:ref2}    
s_{\omega}^{\mathrm{strip}_{L^{\vee}}}\leq s_\omega^\infty.
\end{equation}
Combining the inequalities \eqref{eq:ref1} and \eqref{eq:ref2} gives the desired result. 
\end{proof}

\subsection{Finite books}
\label{sec:finite-books}

In this section, we consider the existence of ground states on finite books. 

The following result is valid for compact books, for which the existence of a minimizer is a direct consequence of Sobolev embeddings and boundedness of the minimizing sequences. However, it takes all its sense for non-compact books, for which convergence of minimizing sequences is more delicate to obtain. 

\begin{lemma}
\label{lemExFinBook}
    Let $\Lib$ be a finite book. 
    If 
    \begin{gather}
        \label{eqOmegaCond}
        \omega > -\omega_{\Lib},
\\
\label{eqGSExisCond}
       s_\omega < s_\omega^\infty,
    \end{gather}
    then there exists a minimizer for \eqref{eqSDefEquiv}.
\end{lemma}

\begin{remark}
We have seen in Lemma \ref{lem:omega_less_omega_L} that 
$\omega \geq -\omega_{\Lib}$ is a necessary condition for the existence of action ground states. The existence in the limit case $\omega = -\omega_{\Lib}$  depends on the considered graph, this is why it is excluded from \eqref{eqOmegaCond}. 
The condition \eqref{eqGSExisCond} is not necessary for existence.
 Indeed, even when $ s_{\omega} = s_{\omega}^{\infty}$, there may still exist a profile $u \in H^1_D(\Lib)$ such that $S_{\omega}(u) = s_\omega$ and $I_\omega(u) = 0$. 
\end{remark}

The \emph{run-away behavior} is a situation where the minimizing sequences of a problem tend to escape towards the infinite directions of the underlying physical space. Such situations are not problematic for homogeneous problems, where the minimizing sequences can simply be shifted back to a compact area thanks to translation invariance. In non-homogeneous situations such as in the presence of a potential, or, as in the present case, for open books, a non-escaping condition such as \eqref{eqGSExisCond} prevents the runaway behavior. 
 In contrast, the condition \eqref{eqOmegaCond} is required to establish a non-trivial lower bound for  $S_\omega$.

\begin{proof}[Proof of Lemma \ref{lemExFinBook}]
Let $(u_n)$ be a minimizing sequence for \eqref{eqSDefEquiv}. Without loss of generality, we may assume that $I_\omega(u_n)=0$. From $I_\omega(u_n) = 0$, we obtain
\begin{equation}\label{eqH1LpUpBnd}
    \|  \nabla u_n \|_{L^2(\Lib)}^{2} + \omega \| u_n \|_{L^2(\Lib)}^{2} = \| u_n \|_{L^{p+1}(\Lib)}^{p+1} \to s_{\omega}/c_p < \infty.
\end{equation}
Thus, $(u_n)$ is bounded in $H^1(\Lib)$, and, up to a subsequence, there exists $u\in H^1_D(\Lib)$ such that $u_n \rightharpoonup u$ weakly in $H^1(\Lib)$. By the non-escaping condition \eqref{eqGSExisCond}, we have $u \not \equiv 0$. 

We now prove that  $I_{\omega}(u) \leq 0$. By the Brezis-Lieb lemma \cite{BrLi83}, we obtain
\[ 
    I_{\omega}(u_n) - I_{\omega}(u_n - u) - I_{\omega}(u) \to 0 \quad \mbox{ as } n \to \infty.
\] 
Assume by contradiction that  $I_{\omega}(u) > 0$. It follows that
\[ 
    \lim_{n \to \infty} I_{\omega}(u_n - u) = \lim_{n \to \infty} I_{\omega}(u_n) - I_{\omega}(u) = -I_{\omega}(u)< 0.
\] 
Thus, there exists $N \in\mathbb N$ such that for any $n > N$, we have $ I_{\omega}(u_n - u) <0.$ Consequently, for $n >N$, we obtain
\begin{equation}
        \label{eqContradict1}
        s_{\omega} < c_p \| u_n - u \|_{L^{p+1}(\Lib)}^{p+1}.
\end{equation} 
    Since $u \not\equiv 0$, the Brezis-Lieb lemma gives
    \[ 
    \lim_{n \to \infty} c_p \| u_n - u \|_{L^{p+1}(\Lib)}^{p+1} = \lim_{n \to \infty} c_p \| u_n \|_{L^{p+1}(\Lib)}^{p+1} -  c_p \| u \|_{L^{p+1}(\Lib)}^{p+1} < s_{\omega}
    \] 
    which contradicts \eqref{eqContradict1}. 
    Consequently, we have $I_{\omega}(u) \leq 0$ and $s_{\omega} \leq S_{\omega}(u)$. 
    On the other hand, by weak lower semicontinuity, we get 
    \[
    c_p \| u \|_{L^{p+1}(\Lib)}^{p+1} \leq \lim_{n \to \infty} c_p \| u_n \|_{L^{p+1}(\Lib)}^{p+1} = s_{\omega},
    \] 
    implying that 
    \[
    c_p \| u \|_{L^{p+1}(\Lib)}^{p+1} = s_{\omega}.
    \]
    Hence $u$ is a non-trivial minimizer for \eqref{eqSDefEquiv}, completing the proof. 
\end{proof}

\subsection{Periodic Books}
\label{sec:periodic}

Periodic books are another important example of possible book designs, e.g in metamaterials. Rigorously, a periodic book can be defined by mimicking one of the definitions used for periodic metric graphs (see \cite{BeKu13,Pa18} and the discussion in \cite{Do19}). Let $\mathcal{L}=\left(\mathcal{P},\mathcal{B}\right)$ be a book, and consider an action of the group $\mathbb Z^n$ (to which we reduce the definition  for simplicity):
\[
(g,x)\in\mathcal{L}\times \mathbb Z^n\to g.x\in\mathcal{L},
\]
which maps pages to pages, bindings to bindings, and preserves the distance. We say that $\mathcal{L}$ is \emph{periodic} if the action is 
\begin{enumerate}
    \item \emph{free}, i.e. if there exists $x\in\Lib$ such that $g.x=x$, then $g=0$;
    \item \emph{discrete}, i.e. for every $x\in\mathcal{L}$, there exists a neighborhood $U$ of $x$ such that $g.x\not\in U$ for any $g\in\mathbb Z^n\setminus\{0\}$;
    \item co-compact, i.e. there exists a compact set $\mathcal{Q}$ such that $\mathcal{L}=\cup_{g\in\mathbb Z^n}g.\mathcal{Q}$.
 \end{enumerate}
 We will call $\mathcal{Q}$ the \emph{fundamental quire} of the periodic book $\mathcal{L}$. In particular, given any $x\in\mathcal{L}$, there exists $g\in\mathbb Z^n$ such that $g.x\in\mathcal{Q}$.

\begin{lemma}
\label{lem:existence-periodic}
    Let $\Lib$ be a periodic book. Assume that 
    \[
    \omega> -\omega_\Lib
    \]
    Then $s_\omega>0$ and \eqref{eqSDefEquiv} admits a non-trivial minimum. 
\end{lemma}

\begin{proof}
Let $(u_n)$ be a minimizing sequence for \eqref{eqSDefEquiv}. We may assume without loss of generality that $I_\omega(u_n)=0$ for any $n\in\mathbb N$. Arguing as in the proof of Lemma \ref{lemExFinBook} (see \eqref{eqH1LpUpBnd}) we know that $(u_n)$ is bounded in $H^1(\Lib)$. Therefore there exists $u\in H^1_D(\Lib)$ such that, up to a subsequence, $u_n \rightharpoonup u$ weakly in $H^1(\Lib)$. 
By Sobolev embedding and Nehari identity, we obtain 
\begin{equation}\label{eqLowBndPotEn}
    \| u_n \|_{L^{p+1}(\Lib)}^{2} \lesssim \| u_n \|_{H^1(\Lib)}^2 \lesssim  \|  \nabla u_n \|_{L^2(\Lib)}^{2} + \omega \| u_n \|_{L^2(\Lib)}^{2}  = \| u_n \|_{L^{p+1}(\Lib)}^{p+1} \lesssim \| u_n \|_{H^1(\Lib)}^{p+1}.
\end{equation}
From this, we deduce that there exists a constant $C >0$, independent of $n$, such that
\begin{equation}\label{eqLowBndPotEn2}
    0 < C \leq \min\{\| u_n \|_{L^{p+1}(\Lib)},\| u_n \|_{H^1(\Lib)}\}.
\end{equation}
This implies in particular that $u_n  \not \to 0$ in $L^{p+1}(\Lib)$. Therefore, by Lemma \ref{lemLions}, there exist $\eps >0$ and $z_n \in \Lib$ such that, up to a subsequence,
\begin{equation*}
   \int_{B(z_n,L_\Lib/4) } |u_n|^2 dx  \geq \eps.
\end{equation*}
Denote by $\mathcal{Q}$ the fundamental quire of the periodic book $\Lib$. By translating $u_n$  if necessary (i.e. replacing $u_n(\cdot)$ by $u_n(g.\cdot)$ for some $g\in\mathbb Z^n$), we may assume that $(z_n)\in\mathcal{Q}$, and since the quire $\mathcal{Q}$ is compact, $(z_n)$ is bounded. Consequently, $u_n|_{\mathcal{Q}} \to u|_{\mathcal{Q}} \not \equiv 0$ in $L^2(\mathcal{Q})$ by the compact embedding $H^1(\mathcal{Q}) \hookrightarrow L^2(\mathcal{Q})$. 

Now we show that $u_n \to u$ strongly in $L^2(\Lib)$. By contradiction, suppose that
\begin{equation*}
    \liminf_{n\to\infty} \| u_n - u \|_{L^{2}(\Lib)}^2 > 0.
\end{equation*}
Let \( \theta \in \mathbb{R} \) and \( (\omega_n) \subset \mathbb{R} \) be such that \( I_\theta(u) = 0 \) and \( I_{\omega_n}(u_n - u) = 0\) for every \( n \). 
By the Brezis–Lieb lemma,  as  $n \to \infty$, we obtain  
\begin{align}
    \omega_n &= \frac{\| u_n - u \|_{L^{p+1}(\Lib)}^{p+1} - \| \nabla u_n - \nabla u \|_{L^{2}(\Lib)}^{2}}{\| u_n - u \|_{L^{2}(\Lib)}^{2}} \nonumber\\ 
    &= \frac{\| u_n \|_{L^{p+1}(\Lib)}^{p+1} - \| \nabla u_n \|_{L^{2}(\Lib)}^{2} - \| u \|_{L^{p+1}(\Lib)}^{p+1} + \| \nabla u \|_{L^{2}(\Lib)}^{2} + o(1)}{\| u_n - u \|_{L^{2}(\Lib)}^{2}} \nonumber\\ 
    & = \frac{\omega \| u_n \|_{L^{2}(\Lib)}^{2} - \theta \| u \|_{L^{2}(\Lib)}^{2} + o(1)}{\| u_n - u \|_{L^{2}(\Lib)}^{2}} =  \omega + (\omega - \theta) \frac{\| u \|_{L^{2}(\Lib)}^{2} }{\| u_n - u \|_{L^{2}(\Lib)}^{2}}+ o(1).\label{eq:omega_n=w+(w-t)}
\end{align}
Applying the Brezis–Lieb lemma again, we obtain  
\begin{align}
    s_\omega & = \lim_{n\to\infty} c_p \| u_n \|_{L^{p+1}(\Lib)}^{p+1} = \lim_{n\to\infty} c_p  (\| u_n - u \|_{L^{p+1}(\Lib)}^{p+1} + \| u \|_{L^{p+1}(\Lib)}^{p+1}) \nonumber\\
    & \geq \liminf_{n\to\infty} s_{\omega_n} + s_\theta.\label{eq:s_w>s_w_n+s_t}
\end{align}
We are going to show a contradiction with 
Lemma \ref{lemSOmMonotone} by discussing the cases $\theta>\omega$, $\theta<\omega$ and $\theta=\omega$. 
If $\theta>\omega$, the contradiction with 
Lemma \ref{lemSOmMonotone} is clear. If $\theta<\omega$, then from \eqref{eq:omega_n=w+(w-t)}, we have
\[
\liminf_{n\to\infty} \omega_n> \omega,
\]
which implies by
Lemma \ref{lemSOmMonotone} that $\liminf_{n\to\infty} s_{\omega_n}>s_{\omega}$, giving again a contradiction. Finally, if $\theta=\omega$, then from from \eqref{eq:omega_n=w+(w-t)} we have $\lim_{n\to\infty}\omega_n=\omega$. In particular, $\liminf_{n\to\infty} s_{\omega_n}>0$, and \eqref{eq:s_w>s_w_n+s_t} enters again in contradiction with Lemma \ref{lemSOmMonotone}. 

Thus $u_n \to u$ in $L^2(\Lib)$ and $L^{p+1}(\Lib)$ by interpolation. From the lower semi-continuity, we obtain $S_\omega(u) \leq s_\omega$. On the other hand, we can prove that  $I_{\omega}(u) \leq 0$ in the same way as in the second part of the proof of Lemma \ref{lemExFinBook} by using the Brezis-Lieb lemma again. Hence, we obtain $S_\omega(u) = s_\omega$ and $u$ is a non-trivial minimizer for \eqref{eqSDefEquiv}, completing the proof. 
\end{proof}

\begin{proof}[Proof of Theorem \ref{thm:existence}]
Theorem \ref{thm:existence} is a direct consequence of Lemma \ref{lemExFinBook} and Lemma \ref{lem:existence-periodic}.  
\end{proof}

\section{Shrinking limit for graph-based books}
\label{secGraphBased}

 In this section, $\Graph$ will denote a finite or periodic book. We analyze the limiting behavior of graph-based books, see Definition \ref{defGraphBase}.
Observe  that when $ \Lib=\GraphTimesZeroL$, then $\omega_\Graph\geq \omega_\Lib$ (with equality in most of the cases that we are considering). We remark that this class includes a wide variety of physically and mathematically relevant examples, such as those illustrated in Figures \ref{fig:classical-graphs} and \ref{fig:grid-torus}.

The main outcome of this section is the proof of Theorem \ref{thmShrinking1}.

\subsection{The rescaled problem}
\label{sec:rescaling}

It is more convenient to work with functions belonging to the same space and to transfer the dependency in $L$ to the Nehari functionals. This is achieved through a rescaling in the second variable. 
For any $u\in H^1(\GraphTimesZeroL)$ we define $v \in H^1(\GraphTimesZeroOne)$ by 
\begin{equation}\label{eqScaling1}
    v(x,y) = u(x,Ly).
\end{equation}
The minimization problem $s_\omega$ becomes
\begin{equation*}
    s_\omega = \inf\{c_p L \| v \|_{L^{p+1}(\GraphTimesZeroOne)}^{p+1}: v \in H^1_D(\GraphTimesZeroOne)\setminus\{0\}, L I_{\omega,L}(v) \leq 0\},
\end{equation*}
where the rescaled Nehari functional $I_{\omega,L}$  is defined on $H^1_D(\GraphTimesZeroOne)$ by
\begin{equation}\label{eqNehariL}
  I_{\omega,L}(v) = \| \partial_x v \|_{L^2(\GraphTimesZeroOne)}^2 + L^{-2} \| \partial_y v \|_{L^2(\GraphTimesZeroOne)}^2 + \omega\|v \|_{L^2(\GraphTimesZeroOne)}^2 -  \|  v \|_{L^{p+1}(\GraphTimesZeroOne)}^{p+1}.
\end{equation}
For future reference, we also introduce here the limit version of this functional 
\begin{equation}\label{eqNehariL0}
  I_{\omega,0}(v)=I_{\omega,\infty}(v) = \| \partial_x v \|_{L^2(\GraphTimesZeroOne)}^2 + \omega\|v \|_{L^2(\GraphTimesZeroOne)}^2 -  \|  v \|_{L^{p+1}(\GraphTimesZeroOne)}^{p+1},
\end{equation}
where it is understood that when we use the notation $I_{\omega,0}(v)$, the function $v\in H^1_D(\GraphTimesZeroOne)$ also verifies $\partial_y v\equiv 0$.
We will work with the following rescaled minimization problem:
\begin{equation}
        \label{eqSOmegaLMin}
    s_{\omega,L} = \inf\{c_p \| u \|_{L^{p+1}(\GraphTimesZeroOne)}^{p+1}: u \in H^1_D(\GraphTimesZeroOne)\setminus\{0\}, I_{\omega,L}(u) \leq 0\}.
\end{equation}
As $L$ tends to $0$ or $\infty$ we consider the limit versions of  $s_{\omega,L}$, which are given by
\begin{align*}
    s_{\omega,0} &= \inf\{c_p \| u \|_{L^{p+1}(\GraphTimesZeroOne)}^{p+1}: u \in H^1_D(\GraphTimesZeroOne)\setminus\{0\},
    \partial_yu\equiv0, %
    I_{\omega,0}(u) \leq 0\}.
\\
    s_{\omega,\infty} &= \inf\{c_p \| u \|_{L^{p+1}(\GraphTimesZeroOne)}^{p+1}: u \in H^1_D(\GraphTimesZeroOne)\setminus\{0\}, I_{\omega,\infty}(u) \leq 0\}.
\end{align*}
The problem $s_{\omega,\infty}$ is the problem when $L\to\infty$ and is not to be confused with the problem at infinity $s_{\omega}^\infty$ defined in \eqref{eqSOmInfinit}.

Observe that we clearly have $s_{\omega,\Graph}=s_{\omega,0}$, where the Nehari level on the graph was defined in \eqref{eqSOmega_G_Min}. 
The Nehari sets corresponding to the previously defined minimization problems will be defined by
\begin{align}
    \mathcal{N}_{L} &= \{ u \in H^1_D(\GraphTimesZeroOne)\setminus\{0\} :  I_{\omega,L}(u) \leq 0\}, 
    \\  
        \mathcal{N}_{0} &= \{ u \in H^1_D(\GraphTimesZeroOne)\setminus\{0\} : \norm{\partial_yu}_{L^2(\GraphTimesZeroOne)}=0,\,I_{\omega,0}(u) \leq 0\}, 
        \\
                \mathcal{N}_{\infty} &= \{ u \in H^1_D(\GraphTimesZeroOne)\setminus\{0\} : I_{\omega,\infty}(u) \leq 0\}, 
        \\
    \mathcal{N}_{\Graph} &= \{ u \in H^1_D(\Graph)\setminus\{0\} :  I_{\omega,\Graph}(u) \leq 0\}.
\end{align}

Given a function $u\in H^1_D(\GraphTimesZeroOne)$, we will want to scale it so that it belongs to various Nehari manifolds. This will be achieved using a scaling factor (similar to $\pi_{\omega}$ defined in \eqref{eqPiOmega}). For $\omega>\omega_\Lib$ and $L\in[0,\infty]$, we define
\begin{equation}
    \label{eqPiDefin}
    \pi_{\omega,L}(u) =  \left( 1 + \frac{I_{\omega,L}(u)}{\| u\|_{L^{p+1}(\GraphTimesZeroOne)}^{p+1}} \right)^{\frac{1}{p-1}}.
\end{equation}
 In particular, we have $I_{\omega,L}(\pi_{\omega,L}(u) u) = 0$ for any $u \in H^1_D(\GraphTimesZeroOne)\setminus\{0\}$.

\subsection{The action level function}

We now give the key properties of the energy levels $s_{\omega,L}$ previously defined. 

\begin{proposition}
\label{propMonoton}
The following assertions hold.
\begin{itemize}
    \item 
    The function $L \mapsto s_{\omega,L}$ is continuous on $[0,\infty)$. 
    \item 
There exists $L_{min}\geq 0$ such that the function $L \mapsto s_{\omega,L}$ is constant on $[0,L_{min}]$, strictly decreasing on $(L_{min},\infty)$ and $\lim_{L\to\infty}s_{\omega,L}=s_{\omega,\infty}=0$. 
    \end{itemize}
\end{proposition}

We will later prove that if $s_{\omega,0}<s_{\omega,0}^\infty$ (where  $s_{\omega,0}^\infty$ is the rescaled problem at infinity, defined in \eqref{eq:s_omega,Linfty}), then $L_{min}>0$. 

\begin{proof}[Proof of Proposition \ref{propMonoton}]
    We first prove that $s_{\omega,L}$ is continuous at $L=0$. 
    Since $\mathcal{N}_0\subset\mathcal{N}_L$ for any $L>0$, we have $s_{\omega,L}\leq s_{\omega,0}$.
    For any $0<L<1$, let $u_L\in H^1_D(\GraphTimesZeroOne)\setminus\{0\}$ be such that $I_{\omega,L}(u_L)=0$ and $s_{\omega,L}\leq c_p\norm{u_L}_{L^{p+1}}^{p+1}<s_{\omega,L}+L$.  We have
    \[
    \norm{\partial_xu_L}_{L^2(\GraphTimesZeroOne)}^2+
    \frac{1}{L^2}\norm{\partial_yu_L}_{L^2(\GraphTimesZeroOne)}^2
    +\omega\norm{u_L}_{L^2(\GraphTimesZeroOne)}^2
    =\norm{u_L}_{L^{p+1}(\GraphTimesZeroOne)}^{p+1}\leq \frac{1}{c_p}(s_{\omega,0}+1),
    \]
    therefore $(u_L)$ is bounded in $H^1(\GraphTimesZeroOne)$ by a constant which depends only on $s_{\omega,0}$ (and $p$). Moreover, we have 
    \begin{equation}
        \label{eqKinEnYL}
         \norm{\partial_yu_L}_{L^2(\GraphTimesZeroOne)}^2\leq\frac{1}{c_p}(s_{\omega,0}+1) L^2.
    \end{equation}
    Consider $\tilde u_L\in H^1_D(\GraphTimesZeroOne)$, the averaged version of $u_L$  along the transverse variable,
     defined for $(x,y)\in \GraphTimesZeroOne$ by
    \begin{equation}
        \label{eq:averaged-version}
        \tilde u_{L}(x,y) = \int_0^1u_{L}(x,z)dz.
    \end{equation}
    By construction, we have $ \partial_y\tilde u_{L}\equiv 0.$
    In what follows, the calculations will be performed assuming sufficient regularity. The end result will be valid by density.
    Let $(x,y)\in\GraphTimesZeroOne$. We have
    \begin{multline*}
       \abs*{ \tilde u_{L}(x,y)- u_{L}(x,y)}
        \leq \int_0^1\abs*{u_{L}(x,z)-u_{L}(x,y)}dz\\= \int_0^1\abs[\Big]{\int_y^z \partial_yu_{L}(x,t)dt}dz
        \leq \int_0^1\int_0^1\abs*{\partial_yu_{L}(x,t)}dtdz
        \leq \norm{\partial_yu_{L}(x)}_{L_y^2(0,1)}.
    \end{multline*}
We get the following Poincar\'e type estimate
\begin{equation} \label{eq:estimate-on-L2-norm-by-gradient-in-y}    
\begin{aligned}
    \norm{\tilde u_{L}- u_{L}}_{L^2(\GraphTimesZeroOne)}^2 &= \int_\Graph\norm{\tilde u_{L}(x)- u_{L}(x)}_{L^2(0,1)}^2dx \\ 
    &\leq \int_\Graph\norm{\partial_yu_{L}(x)}_{L^2(0,1)}^2dx=\norm{\partial_yu_{L}}_{L^2(\GraphTimesZeroOne)}^2.
\end{aligned}
\end{equation}
It follows from Jensen inequality that
    \begin{multline}\label{eqH1Contr1}
\norm{\partial_x \tilde u_{L}}_{L^2(\GraphTimesZeroOne)}^2=
\int_{\GraphTimesZeroOne}|\partial_x \tilde u_{L}(x,y) |^2dydx
\\
\leq \int_{\Graph}\int_0^1\int_0^1|\partial_x u_L(x,z)|^2dzdydx=\int_{\Graph}\int_0^1|\partial_x u_L(x,z)|^2dzdx=\norm{\partial_x u_{L}}_{L^2(\GraphTimesZeroOne)}^2.
\end{multline}
Similar calculations also give
\begin{equation}
    \label{eqLqContr1}
    \norm{\tilde u_{L}}_{L^2(\GraphTimesZeroOne)}^2 \leq \norm{u_{L}}_{L^2(\GraphTimesZeroOne)}^2,
\quad
\norm{\tilde u_{L}}_{L^{p+1}(\GraphTimesZeroOne)}^{p+1} \leq \norm{u_{L}}_{L^{p+1}(\GraphTimesZeroOne)}^{p+1}.
\end{equation}
In particular $(\tilde u_{L}-u_{L})$ is bounded in $H^1(\GraphTimesZeroOne)$, and by interpolation we have 
\begin{equation}
    \label{eqLpConv1}
    \lim_{L\to0}\norm{\tilde u_{L}-u_{L}}_{L^{p+1}(\GraphTimesZeroOne)}=0.
\end{equation}
We have 
\begin{equation}\label{eqConvEq10}
\begin{aligned}
      I_{\omega,0}(\tilde u_{L}) &=
    \norm{\partial_x\tilde u_{L}}_{L^2(\GraphTimesZeroOne)}^2
    +\omega\norm{\tilde u_{L}}_{L^2(\GraphTimesZeroOne)}^2
    -\norm{\tilde u_{L}}_{L^{p+1}(\GraphTimesZeroOne)}^{p+1}
    \\
    & \leq 
        \norm{\partial_x u_{L}}_{L^2(\GraphTimesZeroOne)}^2
+\omega\norm{u_{L}}_{L^2(\GraphTimesZeroOne)}^2
 -\norm{ u_{L}}_{L^{p+1}(\GraphTimesZeroOne)}^{p+1}
    \\ & \quad+ \left( \norm{u_{L}}_{L^{p+1}(\GraphTimesZeroOne)}^{p+1}
         -\norm{\tilde u_{L}}_{L^{p+1}(\GraphTimesZeroOne)}^{p+1}
        \right)
        \\ &\leq I_{\omega,L}(u_{L})-\frac{1}{L^2}\norm{\partial_yu_{L}}_{L^2(\GraphTimesZeroOne)}^2+\left( \norm{u_{L}}_{L^{p+1}(\GraphTimesZeroOne)}^{p+1}
         -\norm{\tilde u_{L}}_{L^{p+1}(\GraphTimesZeroOne)}^{p+1}
        \right).
\end{aligned}
\end{equation}

Here we obtain that 
\begin{equation}
\label{eq:improved-rate}    
\lim_{L\to 0}\frac{1}{L^2}\norm{\partial_yu_{L}}_{L^2(\GraphTimesZeroOne)}^2=0.
\end{equation}
Indeed, assume by contradiction that  there exist $\delta>0$ and $(L_n)\subset(0,\infty)$ such that $L_n\to 0$ and
\[
\frac{1}{L_n^2}\norm{\partial_yu_{L_n}}_{L^2(\GraphTimesZeroOne)}^2>\delta.
\]
Since $\lim_{L\to 0}\left( \norm{u_{L}}_{L^{p+1}(\GraphTimesZeroOne)}^{p+1}
         -\norm{\tilde u_{L}}_{L^{p+1}(\GraphTimesZeroOne)}^{p+1}
        \right)=0$, for $n$ large enough we would have 
        \[
        I_{\omega,0}(\tilde u_{L_n})<-\frac\delta2.
        \]
    This would  imply that 
    \[  \lim_{n\to\infty}\pi_{\omega,0}(\tilde u_{L_n})<1.
    \]
    Moreover, we would have 
    \begin{equation*}
        \begin{aligned}
    s_{\omega,0} &\leq
    c_p\norm{\pi_{\omega,0}(\tilde u_{L_n})\tilde u_{L_n}}_{L^{p+1}(\GraphTimesZeroOne)}^{p+1} \\
    &\leq c_p\pi_{\omega,0}(\tilde u_{L_n})^{p+1}\norm{ u_{L_n}}_{L^{p+1}(\GraphTimesZeroOne)}^{p+1}\leq \pi_{\omega,0}(\tilde u_{L_n})^{p+1}(s_{\omega,L_n}+L_n).
        \end{aligned}
    \end{equation*}
   
    Passing to the limit in the previous inequality would give
    \[
    s_{\omega,0}<\lim_{n\to\infty}s_{\omega,L_n} 
    \]
    which is a contradiction with $ s_{\omega,0}\geq s_{\omega,L}$ for any $L>0$. Hence \eqref{eq:improved-rate} holds. 
    This implies that 
    \begin{equation}\label{eqConvTo0}
        \lim_{L\to 0}I_{\omega,0}(\tilde u_L)=0
    \end{equation}
    Therefore, $\lim_{n\to\infty}\pi_{\omega,0}(\tilde u_{L_n})=1$ and we have
    \[
    s_{\omega,0}\leq \lim_{L\to0}c_p\norm{ u_L}_{L^{p+1}(\GraphTimesZeroOne)}^{p+1}=\lim_{L\to0}s_{\omega,L}.
    \]
    As we always have $s_{\omega,0}\geq s_{\omega,L}$, the above inequality is in fact an equality, and the function $L\to s_{\omega,L}$ is continuous at $L=0$.

Now, we prove that the function $L\to s_{\omega,L}$ is continuous on $(0,\infty)$. 
Let $L \in (0,\infty)$ and $(L_n)\subset(0,\infty)$ such that $L_n \to L$ as $n\to \infty$.  
By definition of $s_{\omega,L_n}$, for any $n\in\mathbb N$, there exists $u_{n}\in H^1_D(\GraphTimesZeroOne)$ such that $I_{\omega,L_n}(u_{n})=0$ and $s_{\omega,L_n}\leq c_p\norm{u_{n}}_{L^{p+1}(\GraphTimesZeroOne)}^{p+1}<s_{\omega,L_n}+|L-L_n|$. 
We claim that $I_{\omega,L}(u_{n})\to 0$ as $n\to\infty$. Observe first that 
 $(u_{n})$ is bounded in $H^1(\GraphTimesZeroOne)$ by construction. 
Observe also that
    \begin{equation*}
        I_{\omega,L_n}(u_{n}) = 0 = I_{\omega,L}(u_{n}) + \frac{L^2 - L_n^2}{L^2 L_n^2} \| \partial_y u_{n} \|_{L^2(\GraphTimesZeroOne)}^2,
    \end{equation*}
which proves the claim by passing to the limit. As a consequence, we have $\pi_{\omega,L}(u_{n})\to 1$ as $n\to\infty$. Therefore,
\[
s_{\omega,L}\leq \liminf_{n\to\infty} c_p\norm{\pi_{\omega,L}(u_{n})u_{n}}_{L^{p+1}(\GraphTimesZeroOne)}^{p+1}\leq \liminf s_{\omega,L_n}.
\]
On the other hand, let $\eps>0$. There exists $u_{\eps}\in H^1_D(\GraphTimesZeroOne)$ such that $I_{\omega,L}(u_{\eps})=0$ and 
$s_{\omega,L}\leq c_p\norm{u_{\eps}}_{L^{p+1}(\GraphTimesZeroOne)}^{p+1}<s_{\omega,L}+\eps$. We have $I_{\omega,L_n}(u_{\eps})\to0$ as $n\to\infty$, and therefore
\[
\limsup_{n\to\infty} s_{\omega,L_n}\leq \limsup_{n\to\infty} c_p\norm{\pi_{\omega,L_n}(u_{\eps})u_{\eps}}_{L^{p+1}(\GraphTimesZeroOne)}^{p+1}\leq 
s_{\omega,L}+\eps.
\]
Since $\eps$ is arbitrary, this implies that 
    \[
    \lim_{n\to\infty} s_{\omega,L_n}=s_{\omega,L},
    \]
    and the function $L\to s_{\omega,L}$ is continuous on $(0,\infty)$.

Let $L_{min}$ be defined by
\[
L_{min}=\sup\{L_*\geq 0:\forall L\in[0,L_*], \, s_{\omega,L}=s_{\omega,0}\}.
\]
We now prove that the function is strictly decreasing on $(L_{min},\infty)$. We first prove that it is decreasing. 
    Let $u\in H^1_D(\GraphTimesZeroL)$. 
    For any $0 < J < K$ we have 
    \begin{equation}\label{eqNehariComp}
        I_{\omega,J}(u) = I_{\omega,K}(u) + \frac{K^2 - J^2}{K^2 J^2} \| \partial_y u \|_{L^2(\GraphTimesZeroOne)}^2 \geq I_{\omega,K}(u),
    \end{equation}
    with strict inequality if $\| \partial_y u \|_{L^2(\GraphTimesZeroOne)}^2 \neq 0$. Let $(u_n)\subset H^1_D(\GraphTimesZeroL)$ be a minimizing sequence for $s_{\omega,J}$. 
    Then $I_K(u_n) \leq I_J(u_n) \leq 0$. By defining $v_n = \pi_{\omega,K}(u_n) u_n$ (so that $I_K(v_n) = 0$), we have $\pi_{\omega,K}(u_n) \leq 1$  and 
    \begin{equation}
    \label{eq:useful-equation}
        s_{\omega,K} \leq c_p \| v_n \|_{L^{p+1}(\GraphTimesZeroOne)}^{p+1} = c_p \pi_{\omega,K}(u_n)^{p+1}\| u_n \|_{L^{p+1}(\GraphTimesZeroOne)}^{p+1} \leq c_p \| u_n \|_{L^{p+1}(\GraphTimesZeroOne)}^{p+1} 
    \end{equation}
    Passing to the limit, we obtain
    $s_{\omega,K}\leq s_{\omega,J}$
  and the function $L\to s_{\omega,L}$ is indeed decreasing on $[0,\infty)$. 

    To prove that it is strictly decreasing, we proceed as follows. Assume by contradiction that there exist $L_{min}<J<K$ such that $L\to s_{\omega,L}$ is constant on $[J,K]$. Let $(u_n)\subset H^1_D(\GraphTimesZeroOne)$ be a minimizing sequence for $s_{\omega,J}$. From \eqref{eqNehariComp}, we infer that $(u_n)$ will also be a minimizing sequence for $s_{\omega,K}$. Moreover, we will have $\lim_{n\to\infty}I_{\omega,J}(u_n)=\lim_{n\to\infty}I_{\omega,K}(u_n)=0$, which  combined with \eqref{eqNehariComp} gives 
    \[
\lim_{n\to\infty}\norm{\partial_yu_n}_{L^2(\GraphTimesZeroOne)}=0.
    \]
    We consider the averaged version $(\tilde u_n)$ of $(u_n)$ as defined in \eqref{eq:averaged-version}.
    We have $\lim_{n\to\infty}I_{\omega,0}(\tilde u_n)=0$. This  implies that $\lim_{n\to\infty}\pi_{\omega,0}(\tilde u_n)=1$. Therefore, we have
    \[
    s_{\omega,0}\leq \lim_{n\to\infty}c_p\norm{\pi_{\omega,0}(\tilde u_n)\tilde u_n}_{L^{p+1}(\GraphTimesZeroOne)}^{p+1}\leq \lim_{n\to\infty}c_p\norm{ u_n}_{L^{p+1}(\GraphTimesZeroOne)}^{p+1}=s_{\omega,J}=s_{\omega,K}.
    \]
    Since we have proved that $L\to s_{\omega,L}$ is decreasing, this implies that $s_{\omega,L}$ is constant on $[0,K]$. This enters in contradiction with the definition of $L_{min}$. Therefore $L\to s_{\omega,L}$ is  strictly decreasing on $(L_{min},\infty)$.

    To analyse what happens when $L\to\infty$, we construct a family of test functions $u_{\lambda,\gamma}$ which will allow us to estimate $s_{\omega,L}$ as $L\to\infty$ and will also serve as a minimizing sequence for $s_{\omega,\infty}$. 

    Let $v\in H^1_D(\Graph)\setminus\{0\}$ and $w\in H^1(\R)\setminus\{0\}$ such that $\supp(w)\subset[0,1]$. For $\lambda,\gamma>1$, we define $u_{\lambda,\gamma}\in H^1_D(\GraphTimesZeroOne)$ by
    \[
    u_{\lambda,\gamma}(x,y)=\lambda v(x) w(\gamma y). 
    \]
    Let $L>0$. We have 
    \begin{multline*}
            I_{\omega,L}(u_{\lambda,\gamma})=
    \frac{\lambda^2}{\gamma}\norm{v_x}_{L^2(\Graph)}^2
    \norm{w}_{L^2(\R)}^2
+    \frac{\gamma\lambda^2}{L^2}\norm{v}_{L^2(\Graph)}^2
    \norm{w_y}_{L^2(\R)}^2
    \\+
    \frac{\lambda^2}{\gamma}\omega\norm{v}_{L^2(\Graph)}^2
    \norm{w}_{L^2(\R)}^2
    -    \frac{\lambda^{p+1}}{\gamma}\norm{v}_{L^{p+1}(\Graph)}^{p+1}
    \norm{w}_{L^{p+1}(\R)}^{p+1}.
        \end{multline*}
Choosing $\gamma = L$ and $\lambda = L^{\frac{1}{2(p+1)}}$, there exists $L_0>0$ such that if $L>L_0$, then
\[
I_{\omega,L}\left(u_{L^{\frac{1}{2(p+1)}},L}\right)<0,
\]
and therefore 
\[
s_{\omega,L}\leq c_p\norm{u_{L^{\frac{1}{2(p+1)}},L}}_{L^{p+1}(\GraphTimesZeroOne)}^{p+1}=\frac{c_p}{\sqrt{L}}\norm{v}_{L^{p+1}(\Graph)}^{p+1}
    \norm{w}_{L^{p+1}(\R)}^{p+1}.
\]
Therefore, $s_{\omega,\infty}=0$  and $\lim_{L\to\infty}s_{\omega,L}=0$.
\end{proof}

\subsection{Minimizers rigidity}
\label{sec:rigidity}
We denote by $s_{\omega,L}^\infty$ the rescaled problem at infinity
    \begin{equation}
    \label{eq:s_omega,Linfty}        
    s_{\omega,L}^\infty = \inf\left\{\liminf_{n \to \infty} 
    c_p\norm{u_n}_{L^{p+1}(\GraphTimesZeroOne)}^{p+1}
    : u_n \rightharpoonup 0 \text{ in } H^1(\GraphTimesZeroOne),u_n\not \equiv 0, I_{\omega,L}(u_n)\leq 0\right\}.
    \end{equation}
    When $L=0$, the sequences $(u_n)$ are in addition required to satisfy $\partial_yu_n\equiv 0$, so that $s_{\omega,0}^\infty=s_{\omega,\Graph}^\infty$.
    
We now prove that if the minimal action level on the graph, $s_{\omega,\Graph}$, is strictly smaller than the corresponding level at infinity $s_{\omega,\Graph}^\infty$, then for sufficiently small $L \geq 0$, there exist minimizers of $s_{\omega,L}$ on the book $\Graph \times [0,1]$, and these minimizers are independent of the transverse variable $y$.

\begin{proposition} 
\label{prop:rigidity}
Let $\omega > - \omega_\Graph$ and assume that $s_{\omega,\Graph} < s_{\omega,\Graph}^\infty$.  
Then $L_{min}$ given by Proposition \ref{propMonoton} satisfies $L_{min} > 0$.  
Moreover, for every $0 \leq L \leq L_{min}$, there exist minimizers of $s_{\omega,L}$, and any minimizer $v_L \in \mathcal{N}_L$ satisfies
\[
\partial_y v_L \equiv 0.
\]
\end{proposition}

\begin{proof}
By Proposition \ref{propMonoton}, the map $L \mapsto s_{\omega,L}$ is continuous and decreasing on $[0,\infty)$. Analogously, one can show that $L \mapsto s_{\omega,L}^\infty$ is continuous.  
Since by assumption $s_{\omega,0} < s_{\omega,0}^\infty$, there exists $L_\infty > 0$ such that 
\[
s_{\omega,L} < s_{\omega,L}^\infty \quad \text{for all } L \in [0,L_\infty).
\]
Consequently, by Theorem \ref{thm:existence}, for each $0 < L < L_\infty$, there exists a minimizer $u_L \in \mathcal{N}_L$ of $s_{\omega,L}$, with $u_L \in D(H)$ (see \eqref{eq:domain}).  
Inside each page, $u_L$ satisfies
\[
-\partial_{xx} u_L - \frac{1}{L^2} \partial_{yy} u_L + \omega u_L - |u_L|^{p-1} u_L = 0.
\]

By classical elliptic regularity theory, $u_L \in W^{3,q}(P)$ for any $q>2$ and for every page $P$ of $\Graph \times [0,1]$.  
We take the duality product with $\partial_{yy} \bar u_L$ and integrate by parts. Integration can be performed page by page, ensuring that boundary terms cancel appropriately.  
The main subtlety arises from the term $\partial_{xx} u_L$, as the other terms can be handled using the Neumann boundary conditions on $u_L$.  

To analyze this, consider 
\[
\langle \partial_{xx} u_L, \partial_{yy} u_L \rangle.
\]
On a page $P = B \times [0,1]$, where the binding $B$ is parametrized as $[a,b]$ (the case $[a, \infty)$ being similar), we compute
\begin{align*}
\Re \int_P \partial_{xx} u_L \, \partial_{yy} \bar u_L \, dx\, dy 
&= \|\partial_{xy} u_L\|_{L^2(\Graph \times [0,1])}^2 \\
&\quad + \Re \int_0^1 \partial_x u_L(b,y) \, \partial_{yy} \bar u_L(b,y) - \partial_x u_L(a,y) \, \partial_{yy} \bar u_L(a,y) \, dy \\
&\quad - \Re \int_a^b \partial_x u_L(x,1) \, \partial_{xy} \bar u_L(x,1) - \partial_x u_L(x,0) \, \partial_{xy} \bar u_L(x,0) \, dx,
\end{align*}
where we first integrated by parts in $x$ and then in $y$.

Observe that the trace operator maps $W^{3,q}(P)$ into $W^{3-\frac{1}{q},q}(B)$ for any $q \geq 2$.  
By the one-dimensional Sobolev embedding for $q>2$, this implies $u_L \in C^2(B)$, so that both $\partial_{yy} u_L$ and $\partial_{xy} u_L$ are well defined on the bindings.

For the second boundary term, we integrate by parts once more to obtain
\begin{multline*}
\int_a^b \partial_x u_L(x,1) \, \partial_{xy} \bar u_L(x,1) \, dx \\
= \partial_x u_L(b,1) \, \partial_y \bar u_L(b,1) - \partial_x u_L(a,1) \, \partial_y \bar u_L(a,1) - \int_a^b \partial_{xx} u_L(x,1) \, \partial_y \bar u_L(x,1) \, dx.
\end{multline*}
Since $u_L \in  C^2(B)$, we have $\partial_y u_L \equiv 0$ on $B \times \{0,1\}$ by continuity, and therefore this term vanishes, as does
\[
\Re \int_a^b \partial_x u_L(x,0) \, \partial_{xy} \bar u_L(x,0) \, dx = 0.
\]

Similarly, for the first boundary term we have
\begin{multline*}
\int_0^1 \partial_x u_L(b,y) \, \partial_{yy} \bar u_L(b,y) \, dy \\
= \partial_x u_L(b,1) \, \partial_y \bar u_L(b,1) - \partial_x u_L(b,0) \, \partial_y \bar u_L(b,0) - \int_0^1 \partial_{xy} u_L(b,y) \, \partial_y \bar u_L(b,y) \, dy,
\end{multline*}
which also vanishes since $\partial_y u_L \equiv 0$. By the same argument,
\[
\Re \int_0^1 \partial_x u_L(a,y) \, \partial_{yy} \bar u_L(a,y) \, dy = 0.
\]
Consequently, we obtain
\begin{equation} \label{eq:hessian_d_y}
\begin{aligned}
0 &= \|\partial_{xy} u_L\|_{L^2(\Graph \times [0,1])}^2 + \frac{1}{L^2} \|\partial_{yy} u_L\|_{L^2(\Graph \times [0,1])}^2 + \omega \|\partial_y u_L\|_{L^2(\Graph \times [0,1])}^2 \\
&\quad - \Re \int_{\Graph \times [0,1]} \partial_y \left( |u_L|^{p-1} u_L \right) \, \partial_y \bar u_L \, dx \, dy.
\end{aligned}
\end{equation}

Now, fix a sequence $(L_n) \subset (0, L_\infty)$ with $L_n \to 0$ as $n \to \infty$.  
Since $(u_{L_n})$ is uniformly bounded in $H^1(\Graph \times [0,1])$, we may extract a subsequence (still denoted $(u_{L_n})$) such that
\[
u_{L_n} \rightharpoonup u_0 \quad \text{weakly in } H^1(\Graph \times [0,1]),
\]
for some limit $u_0$.  
We shall show that
\[
u_{L_n} \to u_0 \quad \text{strongly in } L^{p+1}(\Graph \times [0,1]),
\]
and that $u_0$ is a minimizer of $s_{\omega,0}$.

To this end, we adopt the notation from the proof of Proposition~\ref{propMonoton}.  
Let $\tilde u_{L_n}$ denote the averaged function associated with $u_{L_n}$ as in \eqref{eq:averaged-version}.  
In view of \eqref{eqLpConv1}, it suffices to show that
\[
\tilde u_{L_n} \to u_0 \quad \text{in } L^{p+1}(\Graph \times [0,1]).
\]
From \eqref{eqConvTo0}, we know that
\[
\lim_{n \to \infty} I_{\omega,0}(\tilde u_{L_n}) = 0,
\]
and therefore
\[
\lim_{n \to \infty} \pi_{\omega,0}(\tilde u_{L_n}) = 1.
\]
Hence, the sequence $\pi_{\omega,0}(\tilde u_{L_n}) \, \tilde u_{L_n}$ is minimizing $s_{\omega,0}$ in the limit $L_n \to 0$.

Repeating the compactness argument in the proof of  
Lemma~\ref{lemExFinBook} (for finite books) or Lemma~\ref{lem:existence-periodic} (for periodic books),  
we obtain, possibly after passing to a subsequence and translating, that
\[
\pi_{\omega,0}(\tilde u_{L_n}) \, \tilde u_{L_n} \to u_0 
\quad \text{strongly in } H^1(\Graph \times [0,1]).
\]
Next, using \eqref{eq:hessian_d_y}, we show that 
\[
\partial_y u_{L_n} \equiv 0
\]
for all $n$ sufficiently large. We estimate the last term on the right-hand side of \eqref{eq:hessian_d_y} as
    \begin{multline*}
            \abs*{\Re\int_{\GraphTimesZeroOne}\partial_y\left( |u_{L_n}|^{p-1}u_{L_n}\right)\partial_y\bar u_{L_n} dxdy}
            \leq
p\int_{\GraphTimesZeroOne}|u_{L_n}|^{p-1}\abs{\partial_yu_{L_n}}^2dxdy\\\leq
p\int_{\GraphTimesZeroOne}|u_0|^{p-1}\abs{\partial_yu_{L_n}}^2dxdy+
p\int_{\GraphTimesZeroOne}\left||u_0|^{p-1}-|u_{L_n}|^{p-1}\right|\abs{\partial_yu_{L_n}}^2dxdy.
        \end{multline*}
The first term may be estimated by
    \[
    \int_{\GraphTimesZeroOne}|u_0|^{p-1}\abs{\partial_yu_{L_n}}^2dxdy\leq \norm{u_0}_{L^\infty(\GraphTimesZeroOne)}^{p-1}\norm{\partial_yu_{L_n}}_{L^2(\GraphTimesZeroOne)}^2.
    \]
    For the second term, recall that, given $s,t>0$, we have
    \[
|s^{p-1}-t^{p-1}|\lesssim
\begin{cases}
 |s-t|(s^{p-2}+t^{p-2}),&\text{ when }p\geq 2,\\
 |s-t|^{p-1},&\text{ when }1< p < 2. 
\end{cases}
    \]
    Therefore, using H\"older inequality, when $1< p < 2$ we have
  \[
  \int_{\GraphTimesZeroOne}\abs*{|u_0|^{p-1}-|u_{L_n}|^{p-1}}\abs{\partial_yu_{L_n}}^2dxdy
  \leq 
  \norm{u_0-u_{L_n}}^{p-1}_{L^{p+1}(\GraphTimesZeroOne)}  \norm{\partial_yu_{L_n}}^{2}_{L^{p+1}(\GraphTimesZeroOne)}
  .
  \] 
  When $p\geq 2$, we have
  \begin{multline*}
     \int_{\GraphTimesZeroOne}\abs*{|u_0|^{p-1}-|u_{L_n}|^{p-1}}\abs{\partial_yu_{L_n}}^2dxdy
        \\
        \leq
   \norm{u_0-u_{L_n}}_{L^{p+1}(\GraphTimesZeroOne)} 
   \left(\norm{u_0}_{L^{p+1}(\GraphTimesZeroOne)} ^{p-2}+\norm{u_{L_n}}_{L^{p+1}(\GraphTimesZeroOne)} ^{p-2}\right)
   \norm{\partial_yu_{L_n}}^{2}_{L^{p+1}(\GraphTimesZeroOne)}.
    \end{multline*}
Summarizing, we have established that
\begin{multline*} 
 \abs*{\Re\int_{\GraphTimesZeroOne}\partial_y\left( |u_{L_n}|^{p-1}u_{L_n}\right)\partial_y\bar u_{L_n}dxdy}
            \\\lesssim
            \norm{u_0}_{L^\infty(\GraphTimesZeroOne)}^{p-1}\norm{\partial_yu_{L_n}}_{L^2(\GraphTimesZeroOne)}^2+
            \norm{u_0-u_{L_n}}_{L^{p+1}(\GraphTimesZeroOne)}^{\min(p-1,1)}\norm{\partial_yu_{L_n}}^{2}_{H^1(\GraphTimesZeroOne)}.
\end{multline*}
By construction, for each $x\in\Graph$, the function $\partial_yu_{L_n}(x,\cdot)$ verifies Dirichlet conditions on $[0,1]$. Therefore, by the Poincaré inequality, we have 
\[ 
\| \partial_{y} u_{L_n}(x) \|_{L^2(0,1)}^2 \lesssim \| \partial_{yy} u_{L_n}(x) \|_{L^2(0,1)}^2.
\]
Integrating in $x$ gives
\[ 
\| \partial_{y} u_{L_n} \|_{L^2(\GraphTimesZeroOne)} \lesssim \| \partial_{yy} u_{L_n} \|_{L^2(\GraphTimesZeroOne)}
\]
we have 
\[
\norm{\partial_{y}u_{L_n}}^{2}_{H^1(\GraphTimesZeroOne)}\leq 
C\left(\norm{\partial_{yx}u_{L_n}}^{2}_{L^2(\GraphTimesZeroOne)} + \norm{\partial_{yy}u_{L_n}}^{2}_{L^2(\GraphTimesZeroOne)} \right).
\]
Thus, inserting these estimates into \eqref{eq:hessian_d_y}, we get for some $C >0$
\begin{equation}\label{eqCrucialIneq}
    \begin{aligned} 
    0 & \geq \left(\frac{1}{L_n^2} - C\right) \norm{\partial_{yy}u_{L_n}}_{L^2(\GraphTimesZeroOne)}^2 \\
    &\quad + \left(1 -  \norm{u_0-u_{L_n}}_{L^{p+1}(\GraphTimesZeroOne)}^{\min(p-1,1)} \right) \norm{\partial_{xy}u_{L_n}}_{L^2(\GraphTimesZeroOne)}^2.
\end{aligned}
\end{equation}
 This also implies that 
\[
\pi_{\omega,0}(\tilde u_{L_n})\, \tilde u_{L_n} \longrightarrow u_0
\qquad\text{in } H^1(\GraphTimesZeroOne),
\]
and therefore, by \eqref{eqCrucialIneq}, there exists $N>0$ such that for all $n\geq N$,
\[
\|\partial_{yy} u_{L_n}\|_{L^2(\GraphTimesZeroOne)}
+\|\partial_{xy} u_{L_n}\|_{L^2(\GraphTimesZeroOne)} = 0.
\]
Consequently, for all $n\geq N$, since $\partial_y u_{L_n}$ satisfies Dirichlet boundary conditions, we infer that
\[
\partial_y u_{L_n}\equiv 0.
\]
We now show that this property persists for every $L$ in the whole interval 
\[
0\leq L\leq L_N.
\]
Indeed, by Proposition~\ref{propMonoton} we have 
\[
s_{\omega,L_N} \leq s_{\omega,L}
\qquad\text{for all } L\in [0,L_N],
\]
while, by the definition of $s_{\omega,0}$ and the fact that 
$\partial_y u_{L_N} \equiv 0$, we also have
\[
s_{\omega,L_N} = s_{\omega,L} = s_{\omega,0}.
\]
Suppose by contradiction that there exists some 
$\tilde L \in (0,L_N)$ such that $s_{\omega,\tilde L}$ admits a minimizer 
$v_{\tilde L}$ with
\[
\partial_y v_{\tilde L} \not\equiv 0.
\]
Let $K\in(\tilde L,L_N)$. Then $s_{\omega,K}=s_{\omega,\tilde L}$ and
\[
I_{\omega,K}(v_{\tilde L})
=
I_{\omega,\tilde L}(v_{\tilde L})
+
\left(
\frac{1}{K^2} - \frac{1}{\tilde L^2}
\right)
\|\partial_y v_{\tilde L}\|^2_{L^2(\GraphTimesZeroOne)}
<0.
\]
Hence,
\[
s_{\omega,K}
\leq
c_p \big\|\pi_{\omega,K}(v_{\tilde L})\, v_{\tilde L}\big\|_{L^{p+1}(\GraphTimesZeroOne)}^{p+1}
<
c_p \|v_{\tilde L}\|_{L^{p+1}(\GraphTimesZeroOne)}^{p+1}
=
s_{\omega,\tilde L},
\]
a contradiction.

To conclude, observe first that under the present assumptions, 
a minimizer for $s_{\omega,L}$ exists for every $L\in[0,L_{min}]$.  
Indeed, this is immediate if $L_\infty \geq L_{min}$. 
If instead $L_\infty < L_{min}$, then for any 
$L\in (L_\infty, L_{min})$ we have
\[
s_{\omega,L}=s_{\omega,0}=s^\infty_{\omega,L},
\]
since $L \to s^\infty_{\omega,L}$ is decreasing and bounded below by $s_{\omega,L}$.  
Therefore, any minimizer of $s_{\omega,0}$ (which exists by assumption) is also 
a minimizer of $s_{\omega,L}$ for all such $L$.
 
 It remains to show that $L_{min}=L_N$. We clearly have $L_N\leq L_{min}$. Assume by contradiction that $L_N< L_{min}$. Let $L\in(L_N, L_{min})$ and let $u_L\in\mathcal{N}_L$ be such that 
 \begin{equation*}
     s_{\omega,L} \leq c_p \| u_L\|_{L^{p+1}(\GraphTimesZeroOne}^{p+1} \leq s_{\omega,L} + \eps
 \end{equation*}
 for some $\eps >0$ to be chosen later. 
 $s_{\omega,L}$ such that $\partial_yu_L\not\equiv 0$. This implies (see e.g. \eqref{eqNehariComp}-\eqref{eq:useful-equation})  that
 \[
 I_{\omega,L_{min}}(u_L)<0,\quad s_{\omega,L_{min}}<s_{\omega,L},
 \]
 which is a contradiction with the definition of $L_{min}$. 
 \end{proof}

Finally, we observe that the threshold $L_{min}$ is sharp as it separates the one-dimensional ground states from purely two-dimensional ones.

\begin{proposition}\label{PrpLsharp}
      If there exists a minimizer $v_L$ of $s_{\omega,L}$ for $L>L_{min}$, then it verifies $\partial_yv_L\not\equiv 0$. 
\end{proposition}

\begin{proof}
    The proof follows from Proposition \ref{propMonoton}. Indeed, suppose by contradiction that there exists $L > L_{min}$ such a minimizer $v_L$ of $s_{\omega,L}$  exists and $\partial_y v_L\equiv 0$. Then we get that $s_{\omega,L} < s_{\omega,L_{min}}$ and $I_{\omega,L_{min}}(v_L) = 0$ which contradicts the definition of $s_{\omega,L_{min}}$.
\end{proof}

\begin{proof}[Proof of Theorem \ref{thmShrinking1}]
    Theorem \ref{thmShrinking1} is a direct consequence of Propositions \ref{propMonoton}, \ref{prop:rigidity} and \ref{PrpLsharp}.
\end{proof}


\bibliographystyle{abbrv} 
\bibliography{biblio}

\end{document}